\documentclass{article}

\usepackage[latin1]{inputenc}
\usepackage{amsmath}
\usepackage{mathrsfs}
\usepackage{amsfonts}
\usepackage{amssymb}
\usepackage{graphicx}
\usepackage{amsthm}
\usepackage[toc,page]{appendix}
\usepackage{geometry}
\usepackage[all]{xy}
\usepackage{lmodern}
\usepackage{bbm}
\usepackage[T1]{fontenc}
\usepackage[dvipsnames]{xcolor}
\usepackage{graphics}
\usepackage[colorlinks=true,linkcolor=NavyBlue,urlcolor=RoyalBlue,citecolor=PineGreen,linktocpage=true]{hyperref}
\usepackage{enumerate}

\numberwithin{equation}{section}

\usepackage{titlesec}
\titleformat*{\paragraph}{\itshape\mdseries}

\newtheorem{Theorem}{Theorem}[section]	
\newtheorem*{Theorem*}{Theorem}
\newtheorem{proposition}[Theorem]{Proposition} 
\newtheorem{lemma}[Theorem]{Lemma}
\newtheorem{Claim}[Theorem]{Claim}	
\newtheorem{defn}[Theorem]{Definition}
\newtheorem*{defn*}{Definition}
\newtheorem{problem}[Theorem]{Problem}	
\newtheorem{question}[Theorem]{Question}	
\newtheorem{Cor}[Theorem]{Corollary}

\newtheorem{conjecture}[Theorem]{Conjecture}

\newcommand{\BE}{{\mathbb{E}}}

\newcommand{\BN}{{\mathbb{N}}}

\newcommand{\BP}{{\mathbb{P}}}

\newcommand{\BR}{{\mathbb{R}}}

\newcommand{\BZ}{{\mathbb{Z}}}

\newcommand{\Fe}{{\mathfrak{e}}}

\newcommand{\CG}{{\mathcal{G}}}

\newcommand{\CW}{{\mathcal{W}}}

\newcommand{\CY}{{\mathcal{Y}}}

\newcommand{\ind}{{\mathbbm{1}}}

\newcommand{\one}{{\textbf{1}}}

\newcommand{\om}{{\omega}}

\newcommand{\si}{{\sigma}}

\title{Simplicial spanning trees in random Steiner complexes}

\author{Ron Rosenthal\footnote{Partially supported by ISF grant 771/17 and BSF grant 2018330} ~~and~\, Lior Tenenbaum\footnote{\vspace{0.1cm}Partially supported by ISF grant 771/17}}
  
\begin{document}
	
\maketitle

\begin{abstract}
  A spanning tree $T$ in a graph $G$ is a sub-graph of $G$ with the same vertex set as $G$ which is a tree. In 1981, McKay proved an asymptotic result regarding the number of spanning trees in random $k$-regular graphs. In this paper we prove a high-dimensional generalization of McKay's result for random $d$-dimensional, $k$-regular simplicial complexes on $n$ vertices, showing that the weighted number of simplicial spanning trees is of order $(\xi_{d,k}+o(1))^{\binom{n}{d}}$ as $n\to\infty$, where $\xi_{d,k}$ is an explicit constant, provided $k> 4d^2+d+2$. A key ingredient in our proof is the local convergence of such random complexes to the $d$-dimensional, $k$-regular arboreal complex, which allows us to generalize McKay's result regarding the Kesten-McKay distribution. 
\end{abstract}

 
 \section{Introduction}
	Let $G=(V,E)$ be a graph with vertex set $V$ and edge set $E$ and for a vertex $v\in V$, denote by $\deg(v)$ its degree. $G$ is called a $k$-regular graph, if $\deg(v)=k$ for all $v\in V$. A subgraph $T=(V',E')$ of $G$ is called a \emph{spanning tree} of $G$ if $T$ is an acyclic, connected graph such that $V'=V$. For a graph $G$, denote by $\kappa_1(G)$ the number of spanning trees in it.

	A classical model for random $k$-regular graphs, called the random matching model $\CG_{n,k}$, is defined for $k\geq 1$ and $n\in\BN$ even as the graph with vertex set $[n]:=\{1,2,\ldots,n\}$ and edge set which is the union of $k$ independent and uniformly distributed perfect matching on the set $[n]$. In \cite{McK81b}, McKay proved the following asymptotic result regarding the number of spanning trees in random $k$-regular graphs.

\begin{Theorem}[\cite{McK81b}]\label{Mckay1}	
	Fix $k\geq 3$. Let $(n_i)_{i=1}^\infty\subset\BN$ be a strictly increasing sequence of even numbers and for $i\geq 1$, let $G_i$ be a random graph sampled according to $\CG_{n_i,k}$. Then 
	\[ 
		\sqrt[n_i]{\kappa_1(G_i)} \underset{^{i\rightarrow 	\infty}}{\longrightarrow}  \xi_{1,k}
	\]
in probability, where
	\begin{equation}
	    \xi_{1,k}:= \frac{(k-1)^{k-1}}{(k^2-2k)^{\frac{k-2}{2}}}\,.
	\end{equation}
\end{Theorem} 

In this paper we generalize Theorem \ref{Mckay1} to the context of simplicial complexes. We concentrate on the model of random $(d,k,n)$-uniform Steiner complexes, which for fixed $d,k\in\BN$ and $n$ a $d$-admissible number (see Definition \ref{defn:admissible}), is defined as the union of $k$ i.i.d. $(n,d)$-Steiner systems chosen uniformly at random from all $(n,d)$-Steiner systems, see Section \ref{sec:Results} for further details and additional models. Under the assumption $k> 4d^2+d+2$, we prove that the weighted number of $d$-dimensional spanning trees $\kappa_d(X_i)$ contained in a random $(d,k,n_i)$-Steiner complex $X_i$ satisfies
\begin{equation}\label{intro}
	\sqrt[\binom{n}{d}]{\kappa_d(X_i)}\underset{^{i\rightarrow 	\infty}}{\longrightarrow}  \xi_{d,k}
\end{equation}
in probability, whenever $(n_i)$ is a sequence of $d$-admissible numbers such that $\lim_{i\to\infty}n_i=\infty$, where $\xi_{d,k}$ is an explicit constant and the weights are determined according to the $(d-1)$-homology over $\BZ$ of the $d$-dimensional spanning trees (see Theorem \ref{matrixtree-thm} for further details). 

In order to establish the generalization of Theorem \ref{Mckay1}, we prove two results regarding the structure of uniform random Steiner complexes which are of independent interest. Both results show that in a certain sense, the random complexes $(X_i)$ are close to the $(d,k)$-arboreal complex $T_{d,k}$ introduced in \cite{PR12} as a high-dimensional counterpart of the $k$-regular tree. 

The first result shows that the local structure of $X_i$ converges to that of the arboreal complex $T_{d,k}$, i.e., that for every $r\geq 0$, the $r$-neighborhood of a fixed $(d-1)$-face in $X_i$ is isomorphic to the $r$-neighborhood of any of the $(d-1)$-faces in $T_{d,k}$, with probability tending to $1$ as $i$ tends to infinity. 

Using the local convergence result together with spectral information on the Laplacian of $T_{d,k}$ from \cite{Ro14}, we prove our second result, showing that the eigenvalues of the Laplacian of $X_i$ converge to the spectrum of $T_{d,k}$ in the sense of weak convergence of probability measures (see Section \ref{sec:Results} for further details). This result generalizes a classical result by Kesten and McKay \cite{Ke59,McK83} regarding the limiting spectrum of $k$-regular graphs. 

In order to conclude the proof of \eqref{intro}, we use a high-dimensional variant of the matrix tree theorem, see \cite{Ka83,DKM09}, which relates the weighted number of $d$-dimensional spanning trees of $X_i$ to the eigenvalues of the Laplacian of $X_i$. Then, using our result on the limiting behaviour of the eigenvalues of $X_i$, we are able to conclude \eqref{intro}.

The results in this paper are part of the second author master degree. In particular, the master thesis \cite{Ten20} contains additional details and further discussion. 

\paragraph{Acknowledgements.} The authors are grateful to Alex Lubotzky for fruitful discussions that led to this work. We would also like to thank Antti Knowles, Alan Lew, Zur Luria and Roy Meshulam for their insightful comments. 

  
\section{Results}\label{sec:Results}

\subsection{Preliminaries} \label{subsec:Preliminaries}
	Let $V$ be a non-empty set.  A \emph{simplicial complex} $X$ on a vertex set $V$, is a collection of finite subsets of $V$ that is closed under inclusion, namely, if $\tau \in X$, then $\sigma \in X$ for all $\sigma \subseteq \tau$. The elements of a simplicial complex are called \emph{faces} or \emph{cells}. For a face $\sigma \in X$, define its \emph{dimension} by $\dim(\sigma):=|\sigma|-1$. The dimension of the simplicial complex $X$ is defined as $\dim(X):= {\sup}_{\sigma \in X} \dim(\sigma)$. A simplicial complex of dimension $d$ is abbreviated $d$-complex and a face of dimension $\ell$ is called an $\ell$-face. For a simplicial complex $X$ and $\ell\geq -1$, the collection of $\ell$-faces of $X$ is denoted by $X^\ell= \{\sigma \in X : \dim(\sigma)=\ell\}$, and the number of $\ell$-faces by $f_\ell:=|X^\ell|$. For $\ell\geq 0$,  the \emph{$\ell$-dimensional skeleton} of $X$ is defined to be $ X^{(\ell)}:= \cup_{j=-1}^{\ell} X^j$. We say that $X$ has a complete $\ell$-skeleton if $X$ contains all subsets of $X^0$ of dimension at most $\ell$, i.e. $X^{(j)}=\binom{X^0}{j+1}$ for all $j\leq \ell$, where for non-empty set $A$ and $j\geq 0$, we denote by $\binom{A}{j}$ all subsets of $A$ of size $j$. The complete complex of dimension $d$ on $n$ vertices is denoted by $K_n^{(d)}$. The degree of an $\ell$-face $\sigma$ in a $d$-complex $X$ is defined by $\deg(\sigma)\equiv \deg_X(\sigma):= |\{ \tau \in X^d : \sigma \subset \tau  \}| $. If all the $(d-1)$-faces of a $d$-complex have degree $k$, we say that the simplicial complex is \emph{$k$-regular}. A $d$-complex is called \emph{pure}, if every face in it is contained in at least one $d$-face. If the complex has a complete $(d-1)$-skeleton, this is equivalent to saying that $\deg(\sigma)\geq 1$ for all $\sigma \in X^{d-1}$. Throughout this paper, with the exception of the arboreal complexes (see Section \ref{sec:Results})), we study simplicial complexes, with a finite vertex set which are pure and have a complete $(d-1)$-dimensional skeleton. For future use, for $\sigma\in X$ and $v\in X^0\setminus \sigma$, we introduce the abbreviation $v\sigma:=\{v\}\cup \sigma$. 
	
	Given two complexes $X$ and $Y$, a map $f:X^0\rightarrow Y^0$ is called a \emph{simplicial map}, if $f[\sigma]$ is a face in $Y$ for any $\sigma\in X$, where $f[\sigma]=\{f(v) ~:~ v\in \sigma\}$. In this case the map $f$ induces a map $\hat{f}:X\rightarrow Y$, which is a mapping of sets. If $\hat{f}$ is also a bijection, then $f$ is called a \emph{simplicial isomorphism}.

\paragraph{Oriented faces and upper-Laplacian.} For $\ell\geq 1$, every $\ell$-face $\sigma=\{ \sigma^{0},\ldots,\sigma^{\ell}\} \in X^\ell$ has two possible orientations, corresponding to the possible orderings of its vertices, up to an even permutation. We denote an oriented face by square brackets, and a flip of orientation by an overline. For example, one orientation of $\sigma=\{x,y,z\} $ is $[x,y,z]=[y,z,x]=[z,x,y]$. The other orientation of $\sigma$ is $\overline{[x,y,z]}=[y,x,z]=[x,z,y]=[z,y,x]$. We denote by $X_{\pm}^{\ell}$ the set of oriented $\ell$-faces (so that $|X_{\pm}^{\ell}|=2|X^{\ell}|$ for $\ell\geq1$). We also define $X^0_\pm=X^0$. 

For $\ell \geq 0$, the space of \emph{$\ell$-forms on $X$}, denoted by $\Omega^{\ell}(X)$, is the vector space of skew-symmetric functions on oriented $\ell$-faces over $\BR$
\[
	\Omega^{\ell}(X):= \big\{ f:X_{\pm}^{\ell}\rightarrow\mathbb{R}\,:\,f(\overline{\sigma})=-f(\sigma)\;\forall\sigma\in X_{\pm}^{\ell}\big\}\,.
\]

We endow $\Omega^{\ell}(X)$ with the inner product
\[
	\langle f,g \rangle =\sum_{\sigma\in X^{\ell}}f(\sigma)g(\sigma)\,.
\]
Note that $f(\sigma)g(\sigma)$ is well-defined since its value is independent of the choice of orientation of the $\ell$-face $\sigma$.  If we denote by $X^{\ell}_+$ a set of oriented $\ell$-faces,  containing exactly one orientation for each of the $\ell$-face, then $(\one_\sigma)_{\sigma\in X^\ell_+}$ is an orthonormal basis for $\Omega^\ell(X)$, where for $\sigma\in X^\ell_\pm$, we define
\[
	\one_{\sigma}(\sigma')=\begin{cases}
		1 & \sigma'=\sigma\\
		-1 & \sigma'=\overline{\sigma}\\
		0 & \text{otherwise}
	\end{cases}\,.
\]

The boundary $\partial \sigma$ of the $(\ell+1)$-face $\sigma=\{\sigma^{0},\ldots,\sigma^{\ell+1}\} \in X^{\ell+1}$ is defined as the set of $\ell$-faces $\{\sigma^{0},\ldots,\sigma^{i-1},\sigma^{i+1},\ldots,\sigma^{\ell}\}$ for $0\leq i\leq \ell+1$. An oriented $(\ell+1)$-face $[\sigma^{0},\ldots,\sigma^{\ell+1}]\in X_{\pm}^{\ell+1}$ induces orientations on the $\ell$-faces in its boundary, as follows: the face $\{ \sigma^{0},\ldots,\sigma^{i-1},\sigma^{i+1},\ldots,\sigma^{\ell+1}\}$ is oriented as $(-1)^{i}[\sigma^{0},\ldots,\sigma^{i-1},\sigma^{i+1},\ldots,\sigma^{\ell+1}]$, where we use the notation  $(-1)\tau:=\overline{\tau}$.

The following neighboring relation for oriented faces was introduced in \cite{PR12}: for $\sigma,\sigma'\in X_{\pm}^{d-1}$, define $\sigma$ and $\sigma'$ to be neighbors, denoted $\sigma'\sim\sigma$ (or $\sigma\overset{_X}{\sim}\sigma'$) if there exists an oriented $d$-face $\tau\in X_{\pm}^{d}$ such that both $\sigma$ and $\overline{\sigma'}$ are in the boundary of $\tau$ as oriented faces (see Figure \ref{fig:An-oriented-edge} for an illustration in the case $d=2$).

\begin{figure}[h]
\centering{}\includegraphics{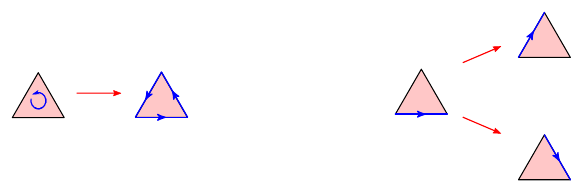}\caption{Left: an oriented 2-face and the orientation it induces on its boundary. Right: an oriented 1-face in a 2-face together with its two oriented neighboring 1-faces.
\label{fig:An-oriented-edge}}
\end{figure}

The adjacency operator $A=A_{X}$ of a complex $X$ is a linear operator $A_X:\Omega^{d-1}(X)\to\Omega^{d-1}(X)$ defined by 
\begin{equation}\label{eq:Adj_operator}
	Af(\sigma):=\sum_{\sigma\overset{X}{\sim}\sigma'}f(\sigma'),\qquad\forall f\in\Omega^{d-1}(X),\,\sigma\in X_{\pm}^{d-1}\,.
\end{equation}
Similarly, the upper Laplacian $\Delta_{d-1}^+=\Delta_{d-1}^+(X):\Omega^{d-1}(X)\to \Omega^{d-1}(X)$ is defined by 
\begin{equation}\label{eq:Adj_operator}
	\Delta_{d-1}^+f(\sigma):=\deg(\sigma)f(\sigma)-\sum_{\sigma'\sim\sigma}f(\sigma'),\qquad\forall f\in\Omega^{d-1}(X),\,\sigma\in X_{\pm}^{d-1}\,,
\end{equation}
where the degree of an oriented face is defined to be the degree of the corresponding unoriented face. As in the graph case $d=1$, The upper Laplacian is a self-adjoint with non-negative eigenvalues. Furthermore, $0$ is always one of its eigenvalues, since each function $g\in \Omega^{d-2}(X)$ defines a $0$-eigenfunction $dg\in \Omega^{d-1}(X)$ by $dg(\sigma)=\sum_{\rho\in\partial \sigma}g(\rho)$. We split the eigenvalues and eigenvectors of $\Delta_{d-1}^+(X)$ into two parts, the trivial $0$ eigenvalues which are the $0$-eigenvalues with eigenvectors of the form $dg$ for some $g\in\Omega^{d-2}(X)$ and the remaining eigenvalues which are called non-trivial. 

The definitions of the adjacency and upper Laplacian that are given here are rather direct. An equivalent, and more conceptual way to define the upper Laplacian, originating in the work of Eckmann \cite{Eck44}, is via the boundary and couboundary opertaros which are closely related to the definition of homology and cohomology over $\BR$. For additional information on the connection between $A,\Delta_{d-1}^{+}$ and the homology and cohomology of the complex c.f. \cite{Ha02,GW14}. 

\paragraph{Simplicial spanning trees.} The next notion we wish to recall is the generalization of a spanning tree for simplicial complexes as introduced in the work of Kalai \cite{Ka83} and of Duval, Klivans and Martin \cite{DKM09}

\begin{defn}[Simplicial spanning trees]
  	Let $X$ be a finite simplicial complex and let $\ell \leq \dim(X)$. An $\ell$-dimensional sub-simplicial complex $T\subseteq X$ is called an \emph{$\ell$-dimensional simplicial spanning tree} of $X$, abbreviated $\ell$-SST, if 
 	\begin{itemize}
  		\item $X^{(\ell-1)}=T^{(\ell-1)}$,
  		\item $\widetilde{H}_\ell(T;\mathbb{Z})=0$,
  		\item $\vert \widetilde{H}_{\ell-1}(T;\mathbb{Z})\vert <\infty$,
  		\item $f_\ell(T)=f_\ell(X)-\widetilde{\beta}_\ell+\widetilde{\beta}_{\ell-1}$,
 	\end{itemize}
where $\widetilde{H}_\ell(T;\mathbb{Z})$ is the $\ell$-th reduced homology group of $T$ with coefficients in $\mathbb{Z}$ and $\widetilde\beta_\ell$ is the $\ell$-th reduced Betti number of $X$, i.e., $\widetilde\beta_\ell=\mathrm{rank}(\widetilde{H}_\ell(X;\mathbb{Z}))$, see \cite{Ha02} for additional information on  homology and \cite{Ka83,DKM09} for more information on SSTs and the reasoning behind the definition.

When $\ell=\dim(X)$, an $\ell$-SST is simply called an SST. Note that we will only use the above definition in the case where the $d$-dimensional complex $X$ has a full $(d-1)$-skeleton. In this case, the the co-dimension 1 skeleton is complete and so is the $(d-1)$-skeleton of each of its SSTs.

The collection of $\ell$-SSTs of $X$ is denoted by $\mathcal{T}_\ell(X)$ and the \emph{weighted number of $\ell$-SSTs} of $X$ is defined by
   \[ 
   		\kappa_\ell(X):= \sum\limits_{T\in \mathcal{T}_\ell(X)}\big\vert \widetilde{H}_{\ell-1}(T;\mathbb{Z}) \big\vert^2. 
   	\]
\end{defn}

\paragraph{Uniform random Steiner complexes.} Next, let us discuss the generalization of the matching model into high-dimensional simplicial complexes. 
  
Let $d\in\BN$ and $n\geq d+1$. An \emph{$(n,d)$-Steiner system} is a collection $S\subset 2^{[n]}$ of subsets of size $d+1$ such that each subset of $[n]$ of size $d$ is contained in exactly $1$ element of $S$. In the language of simplicial complexes, an $(n,d)$-Steiner system can be thought of as the collection of $d$-faces in a $1$-regular $d$-complex with $n$ vertices and complete $(d-1)$-skeleton. In particular an $(n,1)$-Steiner system is a perfect matching. 

Noting that for every $0\leq j\leq d-1$, the number of $(d-1)$-faces containing a fixed $j$-face $\sigma$ in the complete $(d-1)$-complex on $n$ faces is $\binom{n-j-1}{d-j-1}$, and that each $d$-face containing $\sigma$ covers exactly $d-j$ of those $(d-1)$-faces, it follows that if $S$ is an $(n,d)$-Steiner system, then $d-j$ must divide $\binom{n-j-1}{d-j-1}$ for every $0\leq j\leq d-1$. This naturally leads to the following definition.

\begin{defn}[$d$-admissible numbers]\label{defn:admissible}
	For a fixed $d\in\BN$, we say that $n\geq d+1$ is $d$-\textit{admissible} if $d-j$ divides $\binom{n-j-1}{d-j-1}$ for every $0\leq j\leq d-1$. 
\end{defn}	
	
Note that for any fixed $d\in\BN$, there are infinitely many $d$-admissible natural numbers.
 
Although being $d$-admissible is a ncessary condition on $n$ for the existence of an $(n,d)$-Steiner system, the fact that for fixed $d\in\BN$, there are infinitely many $d$-admissible numbers for which Steiner systemes exist is a highly non-trivial fact. This and much more has been proved using a random construction by Peter Keevash \cite{Kee14,Kee18}, who showed that $(n,d)$-Steiner systems exist for any large enough $d$-admissible number $n$. 

The notion of Steiner systems leads us to the following natural generalization of the random matching model.  

\begin{defn}[Uniform random Steiner complexes]\label{defn:random_Steiner_complexes}
	Let $d,k\in\BN$ and $n\in\BN$ a $d$-admissible number. We say that $X$ is a $(d,k,n)$-uniform random Steiner complex if $X=K_{n-1}^{(d-1)}\cup \bigcup_{j=1}^k S_j$, where $(S_j)_{j=1}^k$ are i.i.d. $(n,d)$-Steiner systems sampled uniformly at random from the set of all $(n,d)$-Steiner systems on the vertex set $[n]$. 
\end{defn}

The resulting random complex $X$, is of dimension $d$. Furthermore, the degrees of all $(d-1)$-faces is bounded $k$ and the complex is $k$-regular if and only if the sets $(S_j)_{j=1}^k$ are disjoint (see Section \ref{future} for further discussion). This construction also has the property that it induces the matching model on graphs for links of $(d-2)$-faces and in particular that in dimension $d=1$ it recovers the matching model. Finally, note that the distribution of $X$ is invariant under permutations on the vertex set. 

In \cite{LLR19} a slightly different model, named random Steiner complex, was introduced and studied. There, for $d,k\in\BN$ and $n\in\BN$ which is $d$-admissible, the $(d,k,n)$-random Steiner complex is defined as the union of $k$ independent $(n,d)$-Steiner systems each sampled according to Keevash's construction. That is, if $S_1,\ldots,S_k$ are $(n,d)$-Steiner systems, each of which is sampled independently according to Keevash construction, a $(d,k,n)$-random Steiner complex $X$ is then defined as  $X:=K_n^{(d-1)}\sqcup \bigcup_{j=1}^k S_j$.

In this article, we do not go into the details of Keevash's construction, but take a few useful statements about the algorithm used in its definition. First, with high probability, namely with probability tending to $1$ as the number of vertices tends to infinity, the algorithm produces an $(n,d)$-Steiner system and in particular does not abort. Second, the distribution of the resulting subset of $\binom{[n]}{d+1}$ is invariant under permutations on the vertex set. Finally, it is worth noting that the distribution on Steiner systems obtained from Keevash construction is not the uniform one. 

We note that at the moment there is no algorithm for sampling an $(n,d)$-Steiner system uniformly at random, however Keevash's algorithm provides an algorithm for sampling such systems in a non-uniform way, and it is relatively easy to construct systems which are close to being $(n,d)$Steiner systems, in the sense that all $(d-1)$-cells except for $o(n^d)$ are contained in a unique $d$-cell. See further discussion in Section \ref{future}.

As it turns out the result stated below for uniform random Steiner complexes union are also valid for the original model of random Steiner complexes studied in \cite{LLR19}. Furthermore, all of our results are valid for any distribution on subsets of $A\subset\binom{[n]}{d+1}$ such that 
\begin{enumerate}
	\item Each subset of size $d$ in $[n]$ is contained in at most $1$ element in $A$. 
	\item The probability that $A$ is an $(n,d)$-Steiner system converges to $1$ as $n$ tends to infinity. 
	\item The distribution of $A$ is invariant under permutations of the vertex set. 
\end{enumerate}
Indeed, an inspection of the proofs shows that those are the only properties of the distribution that are used in the proofs of the Theorems.

\paragraph{Arboreal complexes.}

Consider the following construction for an infinite $d$-dimensional complex. Start with a $d$-face $\tau$, and attach to each of its $(d-1)$-faces new $d$-faces, using a new vertex for each of the new $d$-faces. Continue by induction, adding new $d$-faces to each of the $(d-1)$-faces which were added in the last step, using a new vertex for each of them. A complex obtained in such a way is called an arboreal complex. Similar to what happens in the graph case, i.e. $d=1$, for every natural numbers $k$ and $d$, the process in which we add exactly $(k-1)$-new faces to each of the $(d-1)$-faces in each of the steps defines a unique $k$-regular $d$-dimensional arboreal complex, denoted $T_{d,k}$. See Figure \ref{fig:The-orientation-T_2_2}, for an illustration of the first stages in the construction of $T_{2,2}$.
   	
\begin{figure}[h]
	\centering
	\includegraphics[width=13cm]{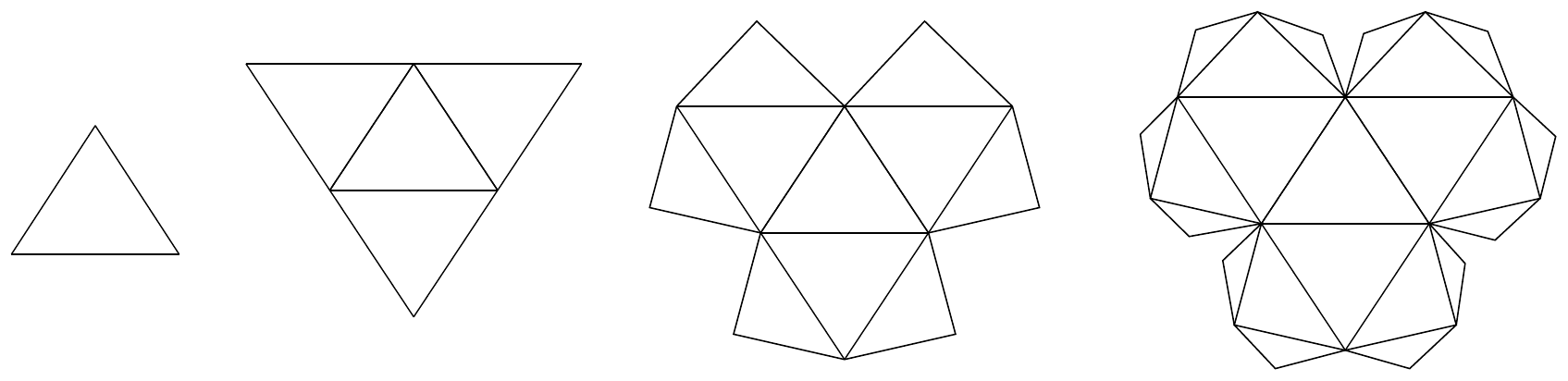}
	\caption{\label{fig:The-orientation-T_2_2} The construction of the zeroth, first, second and third layers of $X=T_{2,2}$.}
\end{figure}

\paragraph{Empirical spectral distribution.}

Let $\CW$ be a $r$-dimensional vector space over $\BR$ and let $A:\CW\rightarrow \CW$ be self-adjoint linear operator on $\CW$ with eigenvalues $\{\lambda_i(A)\}_{i=1}^r$ including multiplicities. The \emph{empirical spectral distribution} of $A$ is the Borel probability measure on $\mathbb{R}$, defined by
 	\[ 
 		\mu_A:= \frac{1}{r} \sum\limits_{i=1}^r \delta_{\lambda_i(A)}\,, 
 	\]
where $\delta_x$ is the Dirac probability measure in $x$.

Throughout the paper, we only discuss the empirical spectral distributions of the adjacency operator $A$ and the upper Laplacian $\Delta_{d-1}^+$ associated with a uniform random Steiner complex $X$, denoted by $\mu_{A_X}$ and $\mu_{\Delta_{d-1}^+(X)}$ respectively. 

	  
\subsection{Results}
  
We now state our main results.  

\begin{Theorem}\label{thm:main_Thm1}	
	Let $(X_i)_{i=1}^\infty$ be a sequence of $(d,k,n_i)$-uniform random Steiner complexes with $(n_i)$ a sequence of $d$-admissible numbers such that $\lim_{i\to\infty}n_i=\infty$, and assume that $k> k(d):= 4d^2+d+2$. Then
\[ 
	\sqrt[ \binom{n_i}{d} ]{ \kappa_d(X_i})\underset{^{i\rightarrow\infty}}{\longrightarrow} \xi_{d,k}
\]
in probability, where
\[
	\xi_{d,k}:= \dfrac{(k-1)^{k-1}}{\left( k-1-d \right)^{ \frac{k}{d+1}-1 } k^{\frac{d(k-1)-1}{d+1}} }\,.
\]
\end{Theorem} 
  
The proof of Theorem \ref{thm:main_Thm1} is based on the asymptotic behaviour of the eigenvalues of uniform random Steiner systems. Recall that a sequence of random probability measures $\mu_n$ converges to a probability measure $\mu$ weakly in probability if $\lim_{n\to\infty}\BP(|\langle \mu_n,f\rangle - \langle \mu,f\rangle|>\varepsilon)=0$, for every $\varepsilon>0$ and every continuous and bounded function $f:\BR\to\BR$, where $\langle \mu_n,f\rangle = \int_\BR f(x)d\mu_n(x)$.
  
\begin{Theorem}\label{thm:main_Thm2}	
  	Let $(X_i)_{i=1}^\infty$ be a sequence of uniform random Steiner complexes with $(n_i)$ a sequence of $d$-admissible numbers such that $\lim_{i\to\infty}n_i=\infty$, and assume that $k\geq d+1$. Then $\mu_{\Delta_{d-1}^+(X_i)}$ converges weakly in probability to $\nu_{d,k}$ as $i\to\infty$, where $\nu_{d,k}$ is the unique Borel probability measure on $\BR$ such that for every Borel set $B\subset \BR$
\[ 
	\nu_{d,k}(B)= \int_{B\cap I_{d,k}}\frac{k\sqrt{4(k-1)d-(k-1+d-x)^2}}{2\pi x((d+1)k-x)}dx\,, 
\]
and
\[ 
	I_{d,k}=\big[(\sqrt{k-1}-\sqrt{d})^2,(\sqrt{k-1}+\sqrt{d})^2]\,.
\]

Similarly, $\mu_{A_{X_i}}$ converges weakly in probability to the probability measure $\mu_{d,k}$ obtained from $\nu_{d,k}$ via the transformation $x\mapsto k-x$ on $\BR$, i.e., for every Borel set $B\subset\BR$
\[
	\mu_{d,k}(B) = \int_{B\cap J_{d,k}}\frac{k\sqrt{4(k-1)d-(x-1+d)^2}}{2\pi(k-x)(dk+x)}dx\,,
\]
where $J_{d,k}=[-d+1-2\sqrt{(k-1)d},-d+1+2\sqrt{(k-1)d}]$.
\end{Theorem}
  
The limiting probability measure $\mu_{d,k}$ from Theorem \ref{thm:main_Thm2} is a high-dimensional variant of the Kesten-McKay distribution $\mu_{1,k}$, which is known to be the spectral measure of the Laplacian of the $k$-regular tree $T_{1,k}=T_{k}$, c.f. \cite{Ke59}. Similarly, it was shown in \cite{Ro14} that $\mu_{d,k}$ is the spectral measure of the upper Laplacian of the arboreal complex $T_{d,k}$.

Finally, the main ingredient in the proof of Theorem \ref{thm:main_Thm2} is a local type convergence result for random $(d,k,n)$-Steiner complexes to $T_{d,k}$. 
\begin{Theorem}\label{thm:main_Thm3}
	Let $(X_i)_{i=1}^\infty$ be a sequence of random $(d,k,n_i)$-Steiner, with $(n_i)$ a sequence of $d$-admissible numbers such that $\lim_{i\to\infty}n_i=\infty$. Then for every $r>0$ the following holds with high probability: The $r$-neighboring complex of a fixed $(d-1)$-face in $X_i$ (see Definition \ref{defn:neighboring_complex}  and Theorem \ref{Ten-Thm3} for precise statements) is isomorphic to the $r$-neighboring complex of any $(d-1)$-face in $T_{d,k}$. 
\end{Theorem}

\paragraph{Conventions.} We use $C$ to denote a generic large positive and finite constant, which may depend on some fixed parameters, and whose value may change from one expression to the next. If $C$ depends on some parameter $k$, we sometimes emphasize this dependence by writing $C_k$ instead of $C$. The letters $d, i, j, k, l, m, n, r,\ell$ are always used to denote an element in $\BN=\{1,2,\ldots\}$ or in $\BN_0=\{0,1,2,\ldots\}$ with $d$ being used to denote the dimension of a complex, $k$ the maximal degree of a $(d-1)$-face in it and $n$ a $d$-admissible number. 
   
  
\section{Local convergence of uniform random Steiner systems} \label{local}

\subsection{$r$-neighboring complex and local convergence}

In this section we define a local structure for simplicial complexes and prove that the local structure of uniform random Steiner complex converges in probability to that of the arboreal complex. The locality is defined with respect to a metric on the $(d-1)$-faces of the complex. Given a $d$-complex $X$, define its \emph{line-graph} $G_d(X)$ to be the graph whose vertex set is $X^{d-1}$ and its edge set $E_d(X)$ is defined as $\{\sigma,\sigma'\}\in X^{d-1}\times X^{d-1}$ such that $\sigma\cup\sigma' \in X^d$. We denote by $\mathrm{dist}_{G_d(X)}$ the graph distance in $G_d(X)$ and for $r\geq 0$ and $\sigma_0\in X^{d-1}$, define $B_r(\sigma_0,X)$ to be the ball of radius $r$ in $G_d(X)$ around $\sigma_0$, namely 
\[
	B_r(\sigma_0,X)=\{\sigma\in X^{d-1} ~:~ \mathrm{dist}_{G_d(X)}(\sigma_0,\sigma)\leq r\}\,.
\]

\begin{defn}[$r$-neighboring complex]\label{defn:neighboring_complex}
	Let $X$ be a $d$-dimensional simplicial complex, $\sigma_0\in X^{d-1}$ and $r\geq 0$. The $r$-neighborhood complex of $\sigma_0$ in $X$, denoted $X(\sigma_0,r)$, is defined to be  the subcomplex $Y\subseteq X$ satisfying
\begin{enumerate}
	\item $Y^{d-1}= B_r(\sigma_0,X)$.
	\item $Y^d=\big\{\tau\in X^d ~:~ \partial\tau\subset Y^{d-1}\big\}$.
	\item $Y^{\ell}=\big\{ \eta\in X^{\ell} ~:~ \eta \subseteq \sigma \; \text{for some} \; \sigma \in Y^{d-1} \big\}$, for every $\ell<d-1$.
\end{enumerate}
\end{defn}

Note that $X(\sigma_0,0)$ is the complex whose faces are subsets of $\sigma_0$, and it contains no $d$-faces. Hence $X(\sigma_0,0)$ is a $(d-1)$-dimensional complex. If $\deg(\sigma_0)=0$, then $X(\sigma_0,r)=X(\sigma_0,0)$ for all $r>0$. On the other hand, if $\deg(\sigma_0)>0$, then $X(\sigma_0,r)$ is $d$-dimensional for all $r>0$. 

For $-1\leq \ell \leq \dim(X(\sigma_0,r))$ and $r\geq 0$, we denote by $X^\ell(\sigma_0,r)$ the $\ell$-dimensional faces of $X(\sigma_0,r)$ and by 
\[	
	\partial X^\ell(\sigma_0,r) = X^\ell(\sigma_0,r)\setminus X^{\ell}(\sigma_0,r-1)
\] 
the $\ell$-faces of $X(\sigma_0,r)$ which are not in $X(\sigma_0,r-1)$, where we define $X^\ell(\sigma_0,-1)=\emptyset$. The latter can be heuristically thought of as the number of $\ell$-faces at distance $r$ from $\sigma_0$.

\begin{Claim} \label{loc-claim}
	Let $X$ be a $d$-complex and $\sigma_0\in X^{d-1}$ such that $\deg(\sigma_0)\geq 1$. Then
\begin{enumerate}
	\item $X(\sigma_0,r)$ is pure for all $r\geq 0$.
	\item If $\sigma \in X^{d-1}$ satisfies $deg_{X}(\sigma)=k$, then $\deg_{X(\sigma_0,r)}(\sigma)=k$ for all $r>\mathrm{dist}_{G_d(X)}(\sigma_0,\sigma)$.
	\item If $v\in \partial X^0(\sigma_0,r)$, then there exists $\sigma \in X^{d-1}(\sigma_0,r-1)$ such that $v\sigma\in \partial X^d(\sigma_0,r)$.
	\item If $v\notin X^0(\sigma_0,r)$, then $\tau \notin X(\sigma_0,r)$ for any $\tau \in X$ satisfying $v\in \tau$. 
	\item For every $r>0$, it holds that $|\partial X^0(\sigma_0,r)|\leq |\partial X^d(\sigma_0,r)|$. Furthermore, if $|\partial X^0(\sigma_0,r)|=|\partial X^d(\sigma_0,r)|$, then each of the $d$-faces in $\partial X^d(\sigma_0,r)$ contains exactly one vertex from $\partial X^0(\sigma_0,r)$. 
\end{enumerate}
\end{Claim}

\begin{proof}$~$
\begin{enumerate}
	\item Note that $X(\sigma_0,0)=\{\sigma : \sigma\subseteq \sigma_0\}$, and hence it is a pure $(d-1)$-dimensional simplicial complex. Assume next that $r>0$. Since $\deg_X(\sigma_0)\geq 1$, by definition $\sigma_0$ is contained in at least one $d$-face in $X$ which is also contained in $X(\sigma_0,r)$. Furthermore, for every $\sigma\in X^{d-1}(\sigma_0,r)\setminus{\sigma_0}$, there exists a sequence $\{\sigma_i\}_{i=1}^r$ such that $\sigma_r=\sigma$ and $\mathrm{dist}_{G_d(X)}(\sigma_{i-1},\sigma_i)=1$ for all $1\leq i\leq d$. In particular $\sigma$ is contained in the $d$-face $\sigma \cup \sigma_{r-1}$. Since all faces of $X(\sigma_0,r)$ of dimension strictly less than $d-1$ are contained is at least one $(d-1)$-faces by definition, we conclude that $X(\sigma_0,r)$ is pure. 

	\item Since $X(\sigma_0,r)$ is a sub-complex of $X$, it is enough to show that every $d$-face containing $\sigma$ in $X$ is also in $X(\sigma_0,r)$. Let $\tau$ be such a $d$-face, and let $\sigma\neq \sigma'\in X^{d-1}$ be a different $(d-1)$ face of $\tau$. Then $\mathrm{dist}_{G_d(X)}(\sigma,\sigma')=1$ and we can conclude by the triangle inequality that $\mathrm{dist}_{G_d(X)}(\sigma_0,\sigma')\leq r$ and hence $\sigma' \in X(\sigma_0,r)$. Since this is true for all $(d-1)$-faces of $\tau$, it follows that $\tau \in X(\sigma_0,r)$ as required.

	\item From the definition of $\partial X^0(\sigma_0,r)$, we know that there exist $\sigma\in \partial X^{d-1}(\sigma_0,r)$ and $\sigma'\in \partial X^{d-1}(\sigma_0,r-1)$ such that $v\in \sigma$, $v\notin \sigma'$ and $\sigma \cup \sigma' \in X^d$. In particular $\tau\equiv \sigma \cup \sigma'=\sigma'\cup \{v\} \in X^d$. Since $\si\in \partial X^{d-1}(\sigma_0,r)$, it follows that $\sigma\notin X^{d-1}(\sigma_0,r-1)$ and hence that $\tau\notin X^d(\sigma_0,r-1)$. In order to show that $\tau\in \partial X^d(\sigma_0,r)$, it remains to show that each of its $(d-1)$-faces is in $X^{d-1}(\sigma_0,r)$. Given a $(d-1)$-face $\sigma''\subset \tau$ distinct from $\sigma$, it follows that $\tau=\sigma\cup\sigma''=\sigma'\cup \{v\}\in X^d$ and therefore, by the triangle inequality, that $\mathrm{dist}_{G_d(x)}(\sigma'',\sigma_0)\leq r$. Hence $\si''\in X^{d-1}(\sigma_0,r)$, as required. 
	
\item This follows from the fact that $X(\sigma_0,r)$ is always a simplicial complex.

\item The first inequality follows from (3.) and the fact that two $d$-faces containing a single vertex in $\partial X^0(\sigma_0,r)$ and a $(d-1)$-face in $X^{d-1}(\sigma_0,r-1)$ are distinct if and only if the $0$-face in $\partial X^0(\sigma_0,r)$ are distinct. If equality holds, then by (3.) we obtain that each of the $d$-faces in $\partial X^d(\sigma_0,r)$ contains a unique vertex from $\partial X^0(\sigma_0,r)$. 
\end{enumerate}
\end{proof}

Next, we turn to define the notion of local convergence. 

\begin{defn}\label{defn_local_confergence}
A pair $(X,\sigma)$, where $X$ is a $d$-complex and $\sigma\in X^{d-1}$ is called a pointed $d$-complex.
\begin{enumerate}
  		\item A sequence of pointed $d$-complexes $\{(X_i,\sigma_i)\}$ is said to locally converge to a pointed $d$-complex $(X_0,\sigma_0)$ if for all $r>0$, there exists $N_r\in \mathbb{N}$, such that $X_i(\sigma_i,r)\cong X_0(\sigma_0,r)$ for all $i>N_r$. 
  		\item A sequence of random pointed $d$-complexes $\{(X_i,\sigma_i)\}$ is said to converge locally in probability to a random pointed simplicial complex $(X_0,\sigma_0)$, if for all $r>0$,
\[
  \lim_{i\to\infty} \mathbb{P}\big( X_i(\sigma_i,r)\not\cong X_0(\sigma_0,r) \big)= 0\,. 
\]
		\item A sequence of random $d$-complexes $\{X_i\}$ is said to converge locally in probability to a deterministic pointed $d$-complex $(X,\sigma_0)$ if the random sequence $(X_i,\sigma_i)$ converges locally to $(X,\sigma_0)$, where given $X_i$, $\sigma_i$ is chosen uniformly at random from $X_i^{d-1}$. In other words, local convergence of unpointed simplicial complexes is always defined with respect to the uniform distribution on the $(d-1)$-faces. 
\end{enumerate}
\end{defn} 

\subsection{Simplicial isomorphism with balls in the arboreal complex}
  
Let $X$ and $Y$ be two complexes and recall the definition of simplicial maps and simplicial isomorphisms from Subsection \ref{subsec:Preliminaries}.
%
The following lemma provides simpler conditions for a map to be simplicial using the notion of a maximal face, namely a maximal element of the complex with respect to inclusion. 

\begin{lemma}\label{iso-lemma}
	Let $X,Y$ be $d$-dimensional simplicial complexes and let $f:X^0\rightarrow Y^0$ be a bijection such that $f[\tau]$ is a face in $Y$ for every maximal face $\tau\in X$. Then $f$ is a simplicial map and $\hat{f}:X\rightarrow Y$ is injective. Furthermore, if $f^{-1}:Y^0\rightarrow X^0$ also satisfies that $f^{-1}[\tau']\in X$ for every maximal face $\tau'\in Y$, then $\hat{f}:X\rightarrow Y$ is a simplicial isomorphism.
\end{lemma}

\begin{proof}
    To show that $\hat{f}$ is a simplicial map, we need to show that $f[\sigma]$ is a face in $Y$ for every face $\sigma\in X$. Since $X$ is finite dimensional, every such face $\sigma$ is contained in some maximal face $\tau$. By our assumption $f[\tau]$ is a face in $Y$, with $f[\sigma]$ as its subset. Since $Y$ is a simplicial complex, we conclude that $f[\sigma]\in Y$. Hence $f$ is simplicial. Also, $\hat{f}$ is injective since $f$ is injective. 
    
    Assume next that $f^{-1}:Y^0\to X^0$ satisfies the additional condition. Due to the first part, all that remains to show is that $\hat{f}$ is surjective. Let $\sigma'\in Y$, and denote its pre-image under $f$ by $\sigma:=f^{-1}[\sigma']$. since $Y$ is finite dimensional, there exists some $\tau'$ which contains $\sigma'$. By assumption we know that $\tau:=f^{-1}[\tau']$ is a face in $X$ with $\sigma$ as its subset. Since $X$ is a simplicial complex, we know that $\sigma$ is a face in $Y$. Thus, we have shown that there exists $\sigma\in X$ such that $\hat{f}(\sigma)=\sigma'$, i.e., that $\hat{f}$ is surjective.
\end{proof}  
  
We are interested in proving local convergence of uniform random Steiner systems to the appropriate arboreal complex. Since the arboreal complex is transitive, i.e., for every $\sigma_1,\sigma_2\in T_{d,k}$ there exists a simplicial automorphism of $T_{d,k}$ taking $\sigma_1$ into $\sigma_2$, it follows that $T_{d,k}(\sigma_1,r)$ and $T_{d,k}(\sigma_2,r)$ are isomorphic for all $r\geq 0$. Hence, we use the abbreviation $T_{d,k}(r)$ to denote any of the above $r$-neighboring complexes in $T_{d,k}$. 
 
\begin{Claim}\label{Ten-Claim1}
	Let $X$ be a finite, pure $d$-complex, $\sigma_0\in X^{d-1}$ and $r_0\geq 0$. Then $X(\sigma_0,r_0) \cong T_{d,k}(r_0)$ if and only if the following conditions hold
	
\begin{enumerate}[(1)]
	\item $|X^0(\sigma_0,r)|= |T^0_{d,k}(r)|$ for all $r\leq r_0$.
	\item $|\partial X^d(\sigma_0,r)|=|\partial T^d_{d,k}(r)|$ for all $r\leq r_0$.
	\item $\deg_X(\sigma)=k$ for all $\sigma \in X^{d-1}$ satisfying $\mathrm{dist}_{G_d(X)}(\sigma,\sigma_0)\leq r_0-1$.
\end{enumerate}
\end{Claim} 

\begin{proof} 
	Since proving that $X(\sigma_0,r_0)\cong T_{d,k}(r_0)$ implies the above conditions is more easily shown, and we only use the other direction, we restrict ourselves to proving that the conditions imply the isomorphism of the complexes. 
	
	We wish to find a bijective simplicial map $\hat{f}_{r_0}:X(\sigma_0,r_0)\rightarrow T_{d,k}(r_0)$ and since both $X(\sigma_0,r_0)$ and $T_{d,k}(r_0)$ are finite $d$-dimensional simplicial complexes, by Lemma \ref{iso-lemma}, it is enough to construct a bijective map $f_{r_0}:X^0(\sigma_0,r_0) \rightarrow T^0_{d,k}(r_0)$ such that $f_{r_0}$ and $ f^{-1}_{r_0}$ preserve maximal faces. We construct such a sequence of maps by induction on $r$ from $0$ to $r_0$. For $r=0$, since both $X(\sigma_0,0)$ and $T_{d,k}(0)$ are composed of a unique $(d-1)$-face and its subsets, i.e. they are both isomorphic to the complete $(d-1)$-dimensional complex $K_d^{(d-1)}$, by transitivity they are isomorphic. In particular, we can fix an arbitrary choice of a simplicial isomorphism $f_0:X^0(\sigma_0,0)\to T^0_{d,k}(0)$. 
	
	Turning to the induction step, assume that for some $0< r\leq r_0$, there exists a simplicial isomorphism $f_{r-1}:X^0(\sigma_0,r-1)\to T_{d,k}^0(r-1)$. From the definition of $T_{d,k}$, each $d$-face in $\partial T_{d,k}^d(\sigma_0,r)$ is composed of a $(d-1)$-face in $\partial T_{d,k}^{d-1}(r-1)$ and a vertex in $\partial T_{d,k}^0(r)$, using a different vertex for each of the $d$-faces. Since the degree of each $(d-1)$-face in $T_{d,k}$ is $k$, we conclude that $|\partial T_{d,k}^d(r)| = (k-1)|\partial T_{d,k}^{d-1}(r-1)|=|\partial T_{d,k}^0(r)|$. By assumptions (1) and (2) together with the previous equality, we get that $|\partial X^{d}(\sigma_0,r)| = |\partial X^0(\sigma_0,r)|$ and hence, using Claim \ref{loc-claim}(5), that each of the $d$-faces in $\partial X^d(\sigma_0,r)$ is composed of a $(d-1)$-face in $\partial X^{d-1}(\sigma_0,r-1)$ and a vertex in $\partial X^0(\sigma_0,r)$. Consequently, for each vertex $v\in \partial X(\sigma_0,r)$, there exists a unique $(d-1)$-face $\sigma$, in $\partial X^{d-1}(\sigma_0,r-1)$ such that $v\sigma\in\partial X^d(\sigma_0,r)$. Furthermore, from assumption (3), for each $(d-1)$-face $\sigma\in \partial X^{d-1}(\sigma_0,r-1)$, the number of vertices $v\in \partial X^k(\sigma_0,r)$ such that $v\sigma\in \partial X^d(\sigma_0,r)$ is exactly $k-1$ for $r>1$ (and $k$ for $r=1$). For every $v\in \partial T_{d,k}^0(r)$ denote by $\sigma_v$ the unique $(d-1)$-face in $\partial T_{d,k}^{d-1}(r-1)$ such that $v\sigma_v\in \partial T_{d,k}^d(r)$. Combining all of the above we conclude that $f_{r-1}:T^0_{d,k}(r-1)\to X^0(\sigma_0,r-1)$ can be extended to a function $f_r:T^0_{d,k}(r)\to X^0(\sigma_0,r)$ such that 
	\begin{itemize}
		\item $f_r|_{T_{d,k}^0(r-1)}=f_{r-1}$ 
		\item $f_r$ is a bijection. 
		\item For every $v\in \partial T_{d,k}^0(r)$, the vertex $f(v)\in X^0$ is one of the $k-1$ ($k$ if $r=0$) vertices in $\partial X^0(\sigma_0,r)$ which belong to a $d$-face in $\partial X^d(\sigma_0,r)$ together with $f[\sigma_v]$.
	\end{itemize}

From the construction $f_r$ is bijective and maps $d$-faces in $\partial T_{d,k}^d(r)$ to $d$-faces in $\partial X^d(\sigma_0,r)$. Since we also assumed that $f_{r-1}$ is a simplicial isomorphism of $T_{d,k}(r-1)$ and $X(\sigma_0,r-1)$ by Lemma \ref{iso-lemma}, this proves that $f_r$ is a simplicial isomorphism of $T_{d,k}(r)$ and $X(\sigma_0,r)$, as required. 
\end{proof} 

Before turning to the proof of Theorem \ref{thm:main_Thm3}, we state a short claim regarding the number of vertices, $(d-1)$-faces and $d$-faces in $T_{d,k}$ and in each of its layers, as defined by the $r$-neighboring complexes.

\begin{Claim} \label{Ten-Claim2} Let $r>1$, then 
\begin{enumerate}
	\item $|T_{d,k}^0(0)|=d$, $|T_{d,k}^{d-1}(0)|=1$ and $|T_{d,k}^d(0)|=0$. 
	\item $|\partial T_{d,k}^0(1)|=k$, $|\partial T_{d,k}^{d-1}(1)|=dk$ and $|\partial T_{d,k}^d(1)|=k$.
	\item $|T_{d,k}^0(r)|=|T_{d,k}^0(r-1)|+|\partial T_{d,k}^d(r)|$.
	\item $|\partial T_{d,k}^d(r)| = (k-1)|\partial T_{d,k}^{d-1}(r-1)|$.
	\item $|\partial T_{d,k}^{d-1}(r)| = d|\partial T_{d,k}^d(r)|$. 
\end{enumerate}
As a result, for all $r\geq 1$
\[ 
	|\partial T_{d,k}^{d-1}(r)|= k(k-1)^{r-1}d^{r} \qquad,\qquad  \quad |\partial T_{d,k}^d(r)|=k(k-1)^{r-1}d^{r-1}  
\]
and 
\[
|T_{d,k}^0(r)|= d+ k\cdot \frac{\big( d(k-1) \big)^{r}-1}{d(k-1)-1} 
\]

\end{Claim}

The claim follows directly from the definition of $T_{d,k}$ and its proof is left to the reader. 

\subsection{Probabilistic estimations for uniform random Steiner complexes}

Claim \ref{Ten-Claim1} can be used to determine whether a sequence of uniform random Steiner complexes converges locally in probability to the arboreal complex $T_{d,k}$. To this end, we wish to accumulate some probabilistic results regarding uniform random Steiner complexes. 

\begin{proposition} \label{Ten-Prop1}
	Let $S$ be an $(n,d)$-Steiner system chosen uniformly at random from the set of all $(n,d)$-Steiner systems. Let $A\subset \binom{[n]}{d+1}$ and $\tau\in\binom{[n]}{d+1}$ such that $\tau\subsetneq A^0:=\bigcup_{\tau'\in A}\tau'$ and $|A^0|\leq \frac{n}{2}$. Then 
\[
	\BP(A\subset S,~\tau\in S)\leq \frac{2}{n}\BP(A\subset S)\,.
\]
\end{proposition}

\begin{proof}
	Denote $\rho=\tau\cap A^0$ and note that by assumption $0\leq |\rho|\leq d$. Since $S$ is a random $(n,d)$-Steiner system sampled uniformly at random for all sufficiently large $n$
	\[
	\begin{aligned}
		\BP(A\subset S) &\geq \BP\bigg(\exists \sigma \in \binom{[n]\setminus A^0}{d+1-|\rho|} ~:~ A\cup \{\rho\sigma\}\subset S\bigg)\\ &= \sum_{\sigma\in \binom{[n]\setminus A^0}{d+1-|\rho|}}\BP(A\subset S,~\rho\sigma \in S) = \binom {n-|A^0|}{d+1-|\rho|}P(A\subset S,~\tau\in S)\,.
	\end{aligned}
	\]
Using the fact that $|\rho|\leq d$ and $|A^0|\leq \frac{n}{2}$, gives $\binom {n-|A^0|}{d+1-|\rho|}\geq n-|A^0|\geq \frac{n}{2}$, thus concluding the proof. 
\end{proof}

\begin{proposition}\label{Ten-Prop3}
	Let $S$ be an $(n,d)$-Steiner system chosen uniformly at random from the set of all $(n,d)$-Steiner systems. Let $A\subset \binom{[n]}{d+1}$ be a family of $d$-cells such that $|A^0|\leq \frac{n}{2}-2d-1$, where $A^0=\bigcup_{\tau\in A}\tau$. Furthermore, let $\sigma,\sigma'\in\binom{[n]}{d}$ be two distinct $(d-1)$-cells and  $v\in [n]\setminus (A^0\cup \sigma\cup\sigma')$ a vertex. Then, there exists $C_d\in (0,\infty)$ such that 
\[
	\BP(A\cup \{v\sigma,v\sigma'\}\subset S)\leq \frac{C_{d}}{n^2}\BP(A\subset S)\,.
\]
\end{proposition}

\begin{proof}
	Denote $\theta= \sigma \cap A^0$, $\theta'=\sigma'\cap A^0$, $\rho=\sigma\setminus A^0$ and $\rho'=\sigma'\setminus A^0$. Throughout the proof we will consider permutation on the vertices which fix $A^0$ and so $\theta$ and $\theta'$ remain fixed while $\rho$ and $\rho'$ varies. Let $j=|\rho\cap\rho'|$, which by assumption satisfies $0\leq j< d$. Due to the symmetry of the model under permutation of the vertex set and the fact that the number of triplets $(\rho_1,\rho_2,u)\in \binom{[n]\setminus A^0}{|\rho|}\times \binom{[n]\setminus A^0}{|\rho'|}\times [n]\setminus A^0$ so that $|\rho_1\cap\rho_2|=j$ and $u\in [n]\setminus (A^0 \cup \rho_1\cup\rho_2)$ is 
\[
M_{n}\equiv M_n(A^0,\rho,\rho'):=\binom{n-|A^0|}{|\rho|}\binom{|\rho|}{j}\binom{n-|A^0|-|\rho|}{|\rho'|-j}(n-|A^0|-|\rho|-|\rho'|+j).
\]
The invariance of the law of $S$ under permutation on the vertices implies that 
\[
\begin{aligned}
	\BP(A\cup \{v\sigma,v\sigma'\}\subset S) 
	&= \frac{1}{M_n}\sum_{\substack{\rho_1\in \binom{[n]\setminus A^0}{|\rho|},\rho_2\in \binom{[n]\setminus A^0}{|\rho'|},\,|\rho_1\cap\rho_2|=j\\u\in [n]\setminus (A^0\cup \rho_1\cup\rho_2)}} 	\BP(A\cup \{\theta u\rho_1,\theta' u\rho_2\}\subset S)\\
&= \frac{1}{M_n}\BE\bigg[\sum_{\substack{\rho_1\in \binom{[n]\setminus A^0}{|\rho|},\rho_2\in \binom{[n]\setminus A^0}{|\rho'|},\,|\rho_1\cap\rho_2|=j\\u\in [n]\setminus (A^0\cup \rho_1\cup\rho_2)}} 	\ind_{(A\cup \{\theta u\rho_1,\theta' u\rho_2\}\subset S}\bigg]\,.
\end{aligned}
\]
Rewriting the sum over $\rho_1,\rho_2$ and $u$ as the number of ways to sample a $j$-cell $\eta$ (that includes the vertex $u$) and two $d$-faces whose intersection is $\eta$, we obtain
\[
\begin{aligned}
	&\BP(A\cup \{v\sigma,v\sigma'\}\subset S) \\
	= &\frac{j+1}{M_n}\BE\bigg[\sum_{\eta\in \binom{[n]\setminus A^0}{j+1}}\sum_{\varrho_1\in\binom{[n]\setminus (A^0\cup \eta)}{|\rho|-j}}\sum_{\varrho_2\in\binom{[n]\setminus (A^0\cup \eta\cup \varrho_1)}{|\rho'|-j}}\ind_{A\cup \{\theta\eta\varrho_1,\theta'\eta\varrho_2\}\subset S}\bigg]\\
	= &\frac{j+1}{M_n}\BE\bigg[\ind_{A\subset S}\sum_{\eta\in \binom{[n]\setminus A^0}{j+1}}\bigg(\sum_{\varrho_1\in\binom{[n]\setminus (A^0\cup \eta)}{|\rho|-j}}\ind_{\theta \eta \varrho_1\in S}\cdot \bigg(\sum_{\varrho_2\in\binom{[n]\setminus (A^0\cup \eta\cup \varrho_1)}{|\rho'|-j}}\ind_{\theta'\eta\varrho_2\in S}\bigg)\bigg)\bigg]\,.\\	
\end{aligned}
\]
Since the number of $(d-1)$-faces in the complex $K_m^{(d-1)}$ containing a fixed cell $\hat{\sigma}$ is $\frac{1}{d+1-|\hat{\sigma}|}\binom{m-|\hat{\sigma}|}{d-|\hat{\sigma}|}$, see the discussion leading the the notion of $d$-admissible numbers, by taking $\hat{\sigma}=\theta'\eta$ and $m=n-|A^0|$ we obtain
\[
	\sum_{\varrho_2\in\binom{[n]\setminus (A^0\cup \eta\cup \varrho_1)}{|\rho'|-j}}\ind_{\theta'\eta\varrho_2\in S}\leq \sum_{\varrho_2\in\binom{[n]\setminus A^0}{|\rho'|-j}}\ind_{\theta'\eta\varrho_2\in S}\leq \frac{1}{d-j-|\theta'|}\binom{n-|A^0|-j-1-|\theta'|}{d-j-1-|\theta'|}
\]
using the last bound together with the bound on the number of $(d-1)$-cells containing $\hat\sigma=\theta\eta$ and $m=n-|A^0|$ gives
\[
\begin{aligned}
	&\sum_{\varrho_1\in\binom{[n]\setminus (A^0\cup \eta)}{|\rho|-j}}\ind_{\theta \eta \varrho_1\in S}\cdot \bigg(\sum_{\varrho_2\in\binom{[n]\setminus (A^0\cup \eta\cup \varrho_1)}{|\rho'|-j}}\ind_{\theta'\eta\varrho_2\in S}\bigg)\\
	\leq & \sum_{\varrho_1\in\binom{[n]\setminus (A^0\cup \eta)}{|\rho|-j}}\ind_{\theta \eta \varrho_1\in S}\cdot \frac{1}{d-j-|\theta'|}\binom{n-|A^0|-j-1-|\theta'|}{d-j-1-|\theta'|}\\
	\leq & \frac{1}{d-j-|\theta|}\binom{n-|A^0|-j-1-|\theta|}{d-j-1-|\theta|}\cdot \frac{1}{d-j-|\theta'|}\binom{n-|A^0|-j-1-|\theta'|}{d-j-1-|\theta'|}\\
	\leq & \binom{n-|A^0|}{d-j-1-|\theta|}\binom{n-|A^0|}{d-j-1-|\theta'|}\,,	
\end{aligned}
\]
where in the last step we used the fact that $\sigma$ and $\sigma'$ are distinct and thus $|\theta\eta|=|\eta|+j+1\leq d$ and $|\theta\eta'|\leq |\theta|+j+1\leq d$.

Summing over the choices for $\eta$ and using $\BE[\ind_{A\subset S}]=\BP(A\subset S)$, gives  
\[
\begin{aligned}
		\BP(A\cup \{v\sigma,v\sigma'\}\subset S)&
	\leq \frac{j+1}{M_n}\binom{n-|A^0|}{d-j-1-|\theta|}\binom{n-|A^0|}{d-j-1-|\theta'|}\binom{n-|A^0|}{j+1}\BP(A\subset S)\\
	&\leq \frac{j+1}{M_n}(n-|A^0|)^{2d-j-1-|\theta|-|\theta'|}\BP(A\subset S)\,.
\end{aligned}
\]
Using the explicit expression for $M_n$, the bounds $\binom{m}{l}\geq \frac{m^l}{l!}$ and $|\rho|,|\rho'|,j\leq d$, and the assumption $n-|A^0|\geq 2d+1$, we conclude that there exists $C_{d,j}\in (0,\infty)$ such that
\[
	M_n\geq \frac{1}{C_{d,j}}(n-|A^0|)^{|\rho|+|\rho'|-j+1}\,.
\]
Combining both estimations and using the fact that $|\rho|+|\theta|=|\rho'|+|\theta'|=d$ and $n-|A^0|\geq \frac{n}{2}$, we conclude that 
\[
	\BP(A\cup \{v\sigma,v\sigma'\}\subset S)\leq \frac{4(j+1)C_{d,j}}{n^2}\BP(A\subset S)\,.
\]
Defining $C_d=4\max\{(j+1)C_{d,j} ~:~ 0\leq j<d\}$ the result follows. 
\end{proof}


\subsection{Proof of Theorem \ref{thm:main_Thm3}}

Using the estimations from the previous subsection we turn to the proof of Theorem \ref{thm:main_Thm3}. The main step in the proof is stated next. 

\begin{Theorem}[Restatement of Theorem \ref{thm:main_Thm3}]\label{Ten-Thm3}
  	Let $X$ be a random $(d,k,n)$-Steiner complexes and $\sigma\in X^{d-1}$. Then for every $r\ge 0$, there exists $C_{r,d,k}\in (0,\infty)$ such that 
\[ 
	\BP\big(X(\sigma,r) \cong T_{d,k}(r)\big) \geq 1-\frac{C_{r,d,k}}{n}\,.
\]
\end{Theorem}

\begin{proof} 
	We prove the statement by induction on $r$. Denote by $S_1,S_2,\ldots,S_k$ the $k$ independent $(n,d)$-Steiner systems used to define $X$. We use the abbreviations
\[
	E_{r} = \{X(\sigma,r) \cong T_{d,k}(r)\}\,.
\]

\paragraph{The case $r=1$.} Since $X(\sigma,0)$ is composed of $\sigma$ and its subsets, by Claim \ref{Ten-Claim1}
\[ 
	E_{1}^c= \{\deg(\sigma)<k \}\,.   
\]
Since $X$ is composed of $k$ independent $(n,d)$-Steiner systems, each of which contains a unique $d$-face that contain $\sigma$, it follows that degree of $\sigma$ is $k$ if and only if those $d$-faces are distinct. Hence
\[ 
	E_{1}^c=\{ \deg(\sigma)<k\} = \bigcup_{1\leq \ell<m\leq k}  \bigcup_{\sigma\subset \tau \in \binom{[n]}{d+1}} \{ \tau \in S_\ell\cap S_m \}\,. 
\]

Using a union bound and the fact that the processes are i.i.d. gives
\begin{equation*}
\begin{aligned}
	\BP(E_{1}^c) &\leq \sum_{1\leq \ell <m \leq k}\sum_{\sigma\subset\tau \in \binom{[n]}{d+1}} \BP(\tau \in S_\ell\cap S_m)\\
	&=\sum_{1\leq \ell <m \leq k}\sum_{\sigma\subset\tau \in \binom{[n]}{d+1}} \BP(\tau \in S_\ell)\cdot \BP(\tau \in S_m)\\
	&= \binom{k}{2} \sum_{\sigma\subset\tau \in \binom{[n]}{d+1}} \BP(\tau \in S_1)^2\,.
\end{aligned}
\end{equation*}
By Proposition \ref{Ten-Prop1} (applied with $A=\emptyset$)
\[ 
	\BP(\tau \in S_1)\leq \frac{C_d}{n}\,, 
\]
and therefore
\[
	\BP(E_{1}^c) \leq \binom{k}{2}\sum_{\sigma\subset\tau \in \binom{[n]}{d+1}} \frac{C_d}{n}\cdot \BP(\tau \in S_1) = \binom{k}{2}\frac{C_d}{n}=\frac{C_{d,k}}{n}\,,
\]
where in the equality before last we used the fact that the events $\{\tau\in S_1\}$ for $\sigma\subset\tau \in \binom{[n]}{d+1}$ are disjoint and their union has probability $1$. This completes the proof for $r=1$. 

\paragraph{General $r$.} We proceed by induction over $r$. Assume that for some $r\geq 1$, there exists $C_{r,d,k}\in (0,\infty)$ such that $\BP(E_{r}^c)\leq C_{r,d,k}n^{-1}$ and let us turn to prove the result for $r+1$. 

By Claim \ref{Ten-Claim1}, for every $r\geq 1$
\begin{equation}\label{eq:r_to_r_plus_one}
\begin{aligned}
	E_{r+1}= E_{r}&\cap \big\{|X^0(\sigma,r+1)|=|T_{d,k}^0(r+1)|\big\}\\ &\cap \big\{|\partial X^d(\sigma,r+1)|=|\partial T_{d,k}^d (r+1)|\big\}\cap \big\{ \deg(\sigma')=k~~ \forall \sigma'\in \partial X^{d-1}(\sigma,r)\big\}\,,
\end{aligned}
\end{equation}
and by the induction assumption 
\[ 
	\BP(E_{r+1}^c) = \BP(E_{r}\cap E_{r+1}^c)+\BP(E_{r}^c\cap E_{r+1}^{c}) 
	 \leq \BP(E_{r}\cap E_{r+1}^c)+\BP(E_{r}^c)\leq \BP(E_{r}\cap E_{r+1}^c)+\frac{C_{r,d,k}}{n}\,.  
\]
Hence, it suffices to estimate the probability of the event $E_{r}\cap E_{r+1}^c$. 

Let $\mathcal{Y}_r$ be the set of all rooted $d$-complexes on the vertex set $[n]$ isomorphic to $T_{d,k}(r)$ whose root is $\sigma$. Then by the law of total probability    
\begin{equation}\label{eq:prob_part1}
	\BP(E_{r}\cap E_{r+1}^c) = \sum_{Y\in \CY_r}\BP(E_{r+1}^c|X(\sigma,r)=Y)\cdot \BP(X(\sigma,r)=Y)\,.
\end{equation}

If $Y\in \CY_r$, then we can use any of the simplicial isomorphism between $Y$ and $T_{d,k}(r)$ to identify the set of $(d-1)$-faces in $Y$ corresponding to $\partial T^{d-1}_{d,k}(r)$, which we denote by $\partial Y$. Note that by \eqref{eq:r_to_r_plus_one} and Claim \ref{Ten-Claim1}, conditioned on the event $X(\sigma,r)=Y$, if the event $E_{r+1}^c$ does not hold, then one of the following must happen:
\begin{enumerate}
	\item There exists $\sigma'\in \partial Y$ such that $\deg(\sigma')<k$. 
	\item There exist $\sigma'\in\partial Y$ and $v\in Y^0\setminus{\sigma'}$ such that $v\sigma'\in \partial X^d(\sigma,r+1)$. 
	\item There exist $\sigma',\sigma''\in \partial Y$ distinct and $v\in [n]\setminus Y^0$ such that $v\sigma',v\sigma''\in \partial X^d(\sigma,r+1)$.
\end{enumerate}
Indeed, if none of the above conditions hold, then the degree of each of the $(d-1)$-faces in $\partial Y$ is $k$, and each of them must be connected to $k$-faces generated by new and distinct vertices. Hence, by Claim \ref{Ten-Claim1}, the event $E_{r+1}$ holds. 

Denoting by $F_{r+1,Y}(\sigma')$, $G_{r+1,Y}(\sigma',v)$ and $H_{r+1,Y}(\sigma',\sigma'',v)$ the three events above respectively, we conclude via a union bound that 
\begin{equation}\label{eq:blablabla}
\begin{aligned}
	&\BP(E_{r+1}^c|X(\sigma,r)=Y) \\
	\leq &\sum_{\sigma'\in \partial Y}\BP(F_{r+1,Y}(\sigma')|X(\sigma,r)=Y)+\sum_{\substack{\sigma'\in \partial Y \\ v\in Y^0\setminus \sigma'}}\BP(G_{r+1,Y}(\sigma',v)|X(\sigma,r)=Y)\\
	+&\hspace{-0.5cm}\sum_{\substack{\sigma',\sigma''\in \partial Y,\, \sigma'\neq \sigma'' \\ v\in [n]\setminus Y^0}}\BP(H_{r+1,Y}(\sigma',\sigma'',v)|X(\sigma,r)=Y)\,.
\end{aligned}
\end{equation}

We turn to estimate each of the sums separately. For the first sum, note that by repeating the argument we used for $r=1$ with $\sigma$ replaced by $\sigma'$ and $A=Y$, we conclude that 
\[
\begin{aligned}
	&\sum_{\sigma'\in \partial Y}\BP(F_{r+1,Y}(\sigma')|X(\sigma,r)=Y) \\ &\qquad \leq |\partial Y|\cdot \binom{k}{2}\frac{C_d}{n} = |\partial T_{d,k}^{d-1}(r)| \binom{k}{2}\frac{C_d}{n} = k(k-1)^{r-1}d^{r}\binom{k}{2}\frac{C_d}{n}\leq \frac{C_{r,d,k}}{n}\,,
\end{aligned}
\]
where in the last equality we used Claim \ref{Ten-Claim2}.

Turning to the second sum, note that by Proposition \ref{Ten-Prop1}, applied with $A=Y^d$, for every $\sigma'\in \partial Y$ and $v\in Y^0\setminus \sigma'$
\[	
	\BP(G_{r+1,Y}(\sigma',v)|X(\sigma,r)=Y) \leq \sum_{m=1}^k \BP(v\sigma'\in S_m) \leq \frac{k\cdot C_d}{n}=\frac{C_{d,k}}{n}\,,
\]
and thus 
\[
	\sum_{\substack{\sigma'\in \partial Y \\ v\in Y^0\setminus \sigma'}}\BP(G_{r+1,Y}(\sigma',v)|X(\sigma,r)=Y)
	\leq |\partial Y|(|Y^0|-d)\frac{C_{d,k}}{n} =|\partial T_{d,k}^{d-1}(r)|\cdot (|T_{d,k}^0(r)|-d)\frac{C_{d,k}}{n}\leq \frac{C_{r,d,k}}{n}\,,
\]
where in the last step we used Claim \ref{Ten-Claim2}.

Finally, turning to estimate the third sum, note that 
\[
	\sum_{\substack{\sigma',\sigma''\in \partial Y \\ v\in [n]\setminus Y^0}}\BP(H_{r+1,Y}(\sigma',\sigma'',v)|X(\sigma,r)=Y) \leq \sum_{\substack{\sigma',\sigma''\in \partial Y \\ v\in [n]\setminus Y^0}}\sum_{1\leq \ell,m\leq k}\BP(v\sigma'\in S_\ell,\, v\sigma''\in S_m|X(\sigma,r)=Y)\,.
\]
We split the sum over $\ell$ and $m$ into the cases $\ell=m$ and $\ell\neq m$. Starting with the former, i.e., with the sum
\begin{equation}\label{eq:blabla}
	\sum_{\substack{\sigma',\sigma''\in \partial Y,\, \sigma'\neq \sigma'' \\ v\in [n]\setminus Y^0}}\sum_{1\leq \ell\leq k}\BP(v\sigma'\in S_\ell,\, v\sigma''\in S_\ell|X(\sigma,r)=Y)\,,
\end{equation}
by Proposition \ref{Ten-Prop3},applied with $A=Y^d$, we have the bound
\[
	\BP(v\sigma'\in S_\ell,\, v\sigma''\in S_\ell|X(\sigma,r)=Y)\leq \frac{C_{d,k}}{n^2}\,.
\]
Since the sum in \eqref{eq:blabla} over $\sigma',\sigma''$ contains at most $|\partial Y|^2$ terms and the sum over $v$ contains at most $n$ terms, we conclude that there exists $C_{r,d,k}\in (0,\infty)$ such that 
\[
	\sum_{\substack{\sigma',\sigma''\in \partial Y,\, \sigma'\neq \sigma'' \\ v\in [n]\setminus Y^0}}\sum_{1\leq \ell\leq k}\BP(v\sigma'\in S_\ell,\, v\sigma''\in S_\ell|X(\sigma,r)=Y)\leq \frac{C_{r,d,k}}{n}\,.
\]
Turning to the latter, i.e., to the sum 
\[
		\sum_{\substack{\sigma',\sigma''\in \partial Y \\ v\in [n]\setminus Y^0}}\sum_{1\leq \ell\neq m\leq k}\BP(v\sigma'\in S_\ell,\, v\sigma''\in S_m|X(\sigma,r)=Y)\,,
\]
we note that by Proposition \ref{Ten-Prop1} for every choice of $\si',\sigma'',\ell$ and $m$ as above
\[
\begin{aligned}
	&\BP(v\sigma'\in S_\ell,\, v\sigma''\in S_m|X(\sigma,r)=Y) \\
	= &\frac{1}{P(X(\sigma,r)=Y)}\sum_{A_1\uplus A_2\uplus\ldots A_k=Y^d} \BP(v\sigma'\in S_\ell,\, v\sigma''\in S_m,~A_i\subset S_i~\forall 1\leq i\leq k)\\
	= &\frac{1}{P(X(\sigma,r)=Y)}\sum_{A_1\uplus A_2\uplus\ldots A_k=Y^d} \BP(A_\ell\cup\{v\sigma'\}\in S_\ell)\BP(A_m\cup\{v\sigma''\}\subset S_m)\prod_{\substack{1\leq i\leq k\\ i\neq \ell,m}}\BP(A_i\subset S_i)\\
	\leq & \frac{1}{P(X(\sigma,r)=Y)}\sum_{A_1\uplus A_2\uplus\ldots A_k=Y^d} \frac{C_d}{n^2}\prod_{1\leq i\leq k}\BP(A_i\subset S_i) = \frac{C_d}{n^2}\,,
\end{aligned}
\]
and thus by repeating the counting argument in the former case we get
\[
	\sum_{\substack{\sigma',\sigma''\in \partial Y \\ v\in [n]\setminus Y^0}}\BP(H_{r+1,Y}(\sigma',\sigma'',v)|X(\sigma,r)=Y) \leq \frac{C_{r,d,k}}{n}\,.
\]

Combining the estimation for the three sums in \eqref{eq:blablabla}, we conclude that 
\[
	\BP(E_{r+1}^c|X(\sigma,r)=Y) \leq \frac{C_{r+1,d,k}}{n}\,,
\]
which together with \eqref{eq:prob_part1} gives 
\[
	\BP(E_{r}\cap E_{r+1}^c)\leq \frac{C_{r+1,d,k}}{n}\sum_{Y\in \CY_r}\BP(X(\sigma,r)=Y)\leq \frac{C_{r+1,d,k}}{n}\,,
\]
thus completing the proof. 
\end{proof}

Using Theorem \ref{Ten-Thm3} we immediately obtain the following corollary

\begin{Cor}\label{cor:local_convergence}
	Let $(X_i)$ be a sequence of random $(d,k,n_i)$-Steiner complexes with $(n_i)$ a sequence of $d$-admissible numbers such that $\lim_{i\to\infty}n_i=\infty$. Then $X_i$ converges locally in probability to $(T_{d,k},\sigma_0)$ for every choice of $\sigma_0\in T_{d,k}^{d-1}$. 
\end{Cor}

\begin{proof}
	Denote by $\sigma_i$ a random element in $X_i^{d-1}$ sampled uniformly at random. Then, by Theorem \ref{Ten-Thm3}, for every $r\geq 0$
	\[
		\BP(X_i(\sigma_i,r)\cong T_{d,k}(r)) = \frac{1}{|X_i^{d-1}|}\sum_{\sigma\in X^{d-1}_i} \BP(X_i(\sigma,r)\cong T_{d,k}(r)) \geq 1-\frac{C_{r,d,k}}{n_i}\,,
	\]
and thus $\lim_{i\to\infty}\BP(X_i(\sigma_i,r)\cong T_{d,k}(r))=1$.
\end{proof}


\section{Weak convergence of the empirical spectral distributions}

The goal of this section is to prove Theorem \ref{thm:main_Thm2}. We start by recalling the definition of weak convergence in probability and state a sufficient and simpler condition for proving it in our setting. Let $(\mu_n)$ be a sequence of random Borel probability measures, recall that $\mu_n$ are said to \emph{converges weakly in probability} to a Borel probability measure $\mu$ on $\BR$ if
\[  
	\lim_{n\to\infty} \BP\big(|\langle \mu_n,f\rangle-\langle \mu,f\rangle| >\varepsilon \big) = 0\,, 
\]
for every continuous and bounded function $f:\BR\rightarrow \BR$ and every $\epsilon>0$, where for every probability measure $\nu$ and every $\nu$-integrable function $f:\mathbb{R}\rightarrow \mathbb{R}$, we abbreviate $\langle \nu,f \rangle :=\int_\mathbb{R} f(x)d\mu(x)$. 

Note that when $\mu_n$ is the empirical spectral distribution of a random self-adjoint operator $A$ with eigenvalues $\lambda_1\leq...\leq \lambda _n$, the term $\langle \mu_n, f\rangle$ is simply the random variable $\langle \mu_n, f\rangle= \frac{1}{n} \sum_{j=1}^n f(\lambda_j)$. In particular, the moments of $\mu_n$ are given by 
\begin{equation} \label{eq:weak_convergence_1}
	\langle \mu_n, x^\ell\rangle= \frac{1}{n} \sum_{j=1}^n \lambda_j^\ell= \frac{1}{n} \mathrm{tr}(A^\ell),\qquad \forall \ell\geq 0\,.
\end{equation}

We have the following useful result providing a sufficient condition for weak convergence in probability. 
\begin{proposition}[\cite{AGZ10} (2.1.8)]\label{wkconv-prop1}
	Let $(\mu_n)$ be a sequence of random Borel probability measures on $\BR$ and $\mu$ a Borel probability measure on $\BR$ such that $\langle \mu_n,x^\ell \rangle \overset{n\rightarrow \infty}{\longrightarrow} \langle \mu,x^\ell \rangle$ in probability for all $\ell \geq 0$. If there exists $K>0$ such that $\mu([-K,K])=1$, then $\mu_n$ converges weakly in probability to $\mu$. 
\end{proposition}

Since the probability measures $\nu_{d,k}$ and $\mu_{d,k}$ from Theorem \ref{thm:main_Thm2} are compactly supported, by Proposition \ref{wkconv-prop1}, it suffices to show that $\langle \mu_{\Delta_{d-1}^+(X_i)},x^\ell\rangle \overset{i\rightarrow \infty}{\longrightarrow}  \langle \nu_{d,k},x^\ell\rangle$ in probability for all $\ell\geq 0$ and similarly $\langle \mu_{A_{X_i}},x^\ell\rangle \overset{i\rightarrow \infty}{\longrightarrow} \langle \mu_{d,k},x^\ell\rangle$. 
In order to prove the weak convergence in probability of the above sequences, we first need a simpler way to describe the limiting values $\langle \nu_{d,k},x^\ell\rangle$ and $\langle \mu_{d,k},x^\ell\rangle$. 

\begin{Theorem}[\cite{Ro14}]\label{Rosen}
	$\nu_{d,k}$ and $\mu_{d,k}$ are the spectral measures of the upper Laplacian $\Delta_{d-1}^+(T_{d,k})$ and the adjacency matrix $A_{T_{d,k}}$ of the arboreal complex $T_{d,k}$ respectively. In particular they are the unique probability measures such that for every $\sigma_0\in T_{d,k}^{d-1}$ and every $\ell\geq 0$
\[
	\langle \Delta_{d-1}^+(T_{d,k})^\ell \one_{\sigma_0},\one_{\sigma_0}\rangle = \langle \nu_{d,k},x^\ell\rangle,\qquad \forall \ell\geq 0
\]
and
\[
	\langle A_{T_{d,k}}^\ell\one_{\sigma_0},\one_{\sigma_0}\rangle = \langle \mu_{d,k},x^\ell\rangle,\qquad \forall \ell\geq 0
\]
respectively.
\end{Theorem}

Combining \eqref{eq:weak_convergence_1}, together with Proposition \ref{wkconv-prop1} and Theorem \ref{Rosen}, we conclude that in order to prove Theorem \ref{thm:main_Thm2} it suffices to show that the following converges in probability for all $\ell\geq 0$
\begin{equation}\label{eq:weak3}
	\binom{n_i}{d}^{-1}\mathrm{tr}(\Delta_{d-1}^+(X_i)^\ell)\underset{^{i\to\infty}}{\longrightarrow}\langle\Delta_{d-1}^+(T_{d,k})^\ell \one_{\sigma_0},\one_{\sigma_0}\rangle
\end{equation}
and 
\begin{equation}\label{eq:weak4}
	\binom{n_i}{d}^{-1}\mathrm{tr}(A_{X_i}^\ell)\underset{^{i\to\infty}}{\longrightarrow}\langle A_{T_{d,k}}^\ell\one_{\sigma_0},\one_{\sigma_0}\rangle\,.
\end{equation}



\subsection{The oriented line-graph}

\begin{defn}\label{orient-line}
	Let $X$ be a $d$-dimensional simplicial complex. The \emph{oriented line-graph} of $X$, denoted $\overrightarrow{G}_d(X)=( X^{d-1}_\pm, \overrightarrow{E}_d(X))$, is the graph whose vertex set $X^{d-1}_\pm$ is composed of all oriented $(d-1)$-faces in $X$ and its edge set $\overrightarrow{E}_d(X)$ is defined to be the set of pairs $\{\sigma,\sigma'\}$ from $X^{d-1}_\pm$ such that $\sigma$ is a neighbor of $\sigma'$ in $X$ (see Figure \ref{fig:An-oriented-edge} for an illustration of the neighboring relation).
\end{defn}

The oriented line-graph allows us to rewrite the left hand-side of \eqref{eq:weak4} in a form which is similar to the term on the right and is thus useful for proving Theorem \ref{thm:main_Thm2}.

\begin{proposition} \label{powers-empirical-dist}
	Let $X$ be a pure, $d$-complex such that $\deg(\sigma)<\infty$ for all $\sigma\in X^{d-1}$. For $\ell\geq 0$ and $\sigma,\sigma'\in X^{d-1}_\pm$ denote by $\phi_\ell(X;\sigma, \sigma')$ the number of paths of length $\ell$ in $\overrightarrow{G}_d(X)$ from $\sigma$ to $\sigma'$. Then 
\[
	\langle A^\ell_X\one_\sigma, \one_{\sigma'}\rangle=\phi_\ell(X;\sigma,\sigma')-\phi_\ell(X;\sigma,\overline{\sigma'}),\qquad\qquad \forall \sigma,\sigma' \in X^{d-1}_\pm\,.	
\]
In particular, for every choice of orientation $X^{d-1}_+$ for each of the $(d-1)$-faces 
\[
	\frac{1}{|X^{d-1}_+|}\mathrm{tr}(A_{X_i}^\ell) =\frac{1}{|X^{d-1}_+|}\sum_{\sigma \in X^{d-1}_+}\langle A^\ell_{X}\one_\sigma , \one_\sigma\rangle= \frac{1}{|X^{d-1}_+|}\sum_{\sigma\in X^{d-1}_+}\big(\phi_\ell(X;\sigma,\sigma)- \phi_\ell(X;\sigma,\overline{\sigma})\big)\,.
\]
\end{proposition} 

\begin{proof}
	The proof follows by induction on $\ell$. For $\ell=0$, $A_X^0=\mathrm{Id}_{\Omega^{d-1}(X)}$ and thus 
\[
	\langle A_X^0 \one_\sigma,\one_{\sigma'}\rangle = \begin{cases} 1 & \sigma=\sigma' \\ 
	-1 & \sigma=\overline{\sigma'}\\	
	0 & \text{otherwise} \end{cases}\,.	
\]
On the other hand, from the definition of $\phi_r$, we have that 
\[
	\phi_0(X;\sigma,\sigma')=\begin{cases} 1 & \sigma=\sigma' \\ 
	0 & \text{otherwise} \end{cases}\,,
\]
and thus 
\[
	\phi_0(X;\sigma,\sigma')-\phi_0(X;\sigma,\overline{\sigma'})=\begin{cases} 1 & \sigma=\sigma' \\ 
	-1 & \sigma=\overline{\sigma'}\\ 0 & \text{otherwise} \end{cases}\,,
\]
which proves the result for $\ell=0$. 

Turning to the induction step, assume the result holds for $\ell$ and observe $\langle A^{\ell+1}_X\one_\sigma,\one_{\sigma'}\rangle$. By definition $\langle A^{\ell+1}_X\one_\sigma,\one_{\sigma'}\rangle = \langle A^{\ell}_X(A_X\one_\sigma),\one_{\sigma'}\rangle$ and since
\[
	A_X\one_\si(\sigma'') = \sum_{\rho\sim\sigma''}\one_\sigma(\rho) = \begin{cases}
	1 & \sigma''\sim\sigma\\
	-1 & \sigma''\sim\overline{\sigma}\\
	0 & \text{otherwise}
	\end{cases}\,,
\]
namely
\[
	A_X\one_\si = \sum_{\rho\sim\sigma}\one_{\rho}\,,
\]
we conclude that 
\[
	\langle A^{\ell+1}_X\one_\sigma,\one_{\sigma'}\rangle = \big\langle A^\ell_X\big(\sum_{\rho\sim\sigma}\one_\rho\big),\one_{\sigma'}\big\rangle=\sum_{\rho\sim\sigma}\langle A_X^\ell\one_\rho,\one_{\sigma'}\rangle\,.
\]
Thus by induction
\[
	\langle A^{\ell+1}_X\one_\sigma,\one_{\sigma'}\rangle = \sum_{\rho\sim\sigma}\big(\phi_\ell(X;\rho,\sigma')-\phi_\ell(X;\rho,\overline{\sigma'})\big)=\phi_{\ell+1}(X;\sigma,\sigma')-\phi_{\ell+1}(X;\sigma,\overline{\sigma'})\,,
\]
where in the last step we used the fact that any path of length $\ell+1$ from $\sigma$ to $\sigma'$ is composed of one step from $\sigma$ to a neighbor $\rho$ of $\sigma$ in $\overrightarrow{G}_d(X)$ followed by a path of length $\ell$ from $\rho$ to $\sigma'$.

The formula for $\langle\mu_{A_X},x^\ell\rangle$ follows from the fact that $(\one_\sigma)_{\sigma\in X^{d-1}_+}$ is an orthonormal basis for $\Omega^{d-1}(X)$. 
\end{proof}

\subsection{Proof of Theorem \ref{thm:main_Thm2}}

We start by proving the results for the adjacency matrices. By Proposition \ref{wkconv-prop1} and the fact that $\mu_{d,k}$ is compactly supported it is enough to show that for all $\ell\geq 0$
\[
	\langle \mu_{A_{X_i}},x^\ell\rangle \overset{i\rightarrow \infty}{\longrightarrow} \langle \mu_{d,k},x^\ell\rangle\,,
\]
where the convergence is in probability. Furthermore, by Theorem \ref{Rosen}, this is equivalent to proving \eqref{eq:weak4}, i.e., that for every $\ell\geq 0$ 
\[
		\binom{n_i}{d}^{-1}\mathrm{tr}(A_{X_i}^\ell)\underset{^{i\to\infty}}{\longrightarrow}\langle A_{T_{d,k}}^\ell\one_\sigma,\one_\sigma\rangle,\qquad \text{in probability}\,.
\]

Finally, by Proposition \ref{powers-empirical-dist} this is equivalent to proving that for every $l\geq 0$
\[
	\frac{1}{|X^{d-1}_{i,+}|}\sum_{\sigma\in X^{d-1}_{i,+}}\big(\phi_\ell(X_i;\sigma,\sigma)- \phi_\ell(X_i;\sigma,\overline{\sigma})\big)\underset{^{i\to\infty}}{\longrightarrow} \phi_\ell(T_{d,k},\sigma_0,\sigma_0)-\phi_\ell(T_{d,k};\sigma_0,\overline{\sigma_0})\,,
\]
where the convergence is in probability and $\sigma_0$ is some arbitrary choice of an oriented $(d-1)$-face in $T_{d,k}$.

Abbreviate 
\[
	\Phi_{i,\ell}:=\sum_{\sigma\in X^{d-1}_{i,+}}\big(\phi_\ell(X_i;\sigma,\sigma)- \phi_\ell(X_i;\sigma,\overline{\sigma})\big)\,.
\]

For $i\geq 1$ and $\ell\geq 0$, denote by $N_{i,\ell}$ the number of $(d-1)$-faces in $X_i^{d-1}$ whose $\lceil \ell/2\rceil$-neighboring complex is isomorphic to $T_{d,k}(\lceil \ell/2\rceil)$, i.e., 
\begin{equation}\label{eq:The_definition_of_N}
	N_{i,\ell/2} = |\{\sigma\in X^{d-1}_i ~:~ X_i(\sigma,\lceil \ell/2\rceil)\cong T_{d,k}(\lceil \ell/2\rceil)\}|\,.
\end{equation}

Since any closed path of length $\ell$ in $\overrightarrow{G}_d(X_i)$ starting from $\sigma\in X^{d-1}_\pm$ uses only vertices which are at distance at most $\lceil \ell/2\rceil$ from $\sigma$ in $\overrightarrow{G}_d(X_i)$, it follows that on the event $E_{i,\lceil \ell/2\rceil}(\sigma)=\{X_i(\sigma,\lceil \ell/2\rceil)\cong T_{d,k}(\lceil \ell/2\rceil)\}$, it holds that 
\[
	\phi_\ell(X_i;\sigma,\sigma)-\phi_\ell(X_i;\sigma,\overline{\sigma}) = \phi_\ell(T_{d,k};\sigma_0,\sigma_0)-\phi_\ell(T_{d,k};\sigma_0,\overline{\sigma_0})\,.
\]
Furthermore, since the degree of each of the $(d-1)$-faces in $X_i$ is bounded by $k$, it follows that the degree of each edge in $\overrightarrow{G}_d(X_i)$ is bounded by $dk$ and hence that for every $\ell\geq 0$ and every $\sigma\in X_{i,\pm}^{d-1}$
\[
	|\phi_\ell(X_i;\sigma,\sigma)-\phi_\ell(X_i;\sigma,\overline{\sigma})|\leq (dk)^\ell\,.
\]
Hence 
\begin{equation}\label{eq:weak_conv_proof_1}
	\Phi_{i,\ell} \geq N_{i,\ell/2}\cdot \big(\phi_\ell(T_{d,k};\sigma_0,\sigma_0)-\phi_\ell(T_{d,k};\sigma_0,\overline{\sigma_0})\big)-(|X^{d-1}_{i,\pm}|-N_{i,\ell/2})\cdot (dk)^\ell\,.
\end{equation}
and
\begin{equation}\label{eq:weak_conv_proof_2}
	\Phi_{i,\ell}\leq N_{i,\ell/2}\cdot \big(\phi_\ell(T_{d,k};\sigma_0,\sigma_0)-\phi_\ell(T_{d,k};\sigma_0,\overline{\sigma_0})\big) + (|X^{d-1}_{i,\pm}|-N_{i,\ell/2})\cdot (dk)^\ell
\end{equation}

Combining \eqref{eq:weak_conv_proof_1} and \eqref{eq:weak_conv_proof_2}, we conclude that 
\begin{equation}\label{eq:weak_conv_proof_3}
\begin{aligned}
	&\bigg|\frac{\Phi_{i,\ell}}{|X_{i,+}^{d-1}|}-\big(\phi_\ell(T_{d,k};\sigma_0,\sigma_0)-\phi_\ell(T_{d,k};\sigma_0,\overline{\sigma_0})\big)\bigg|\\
	\leq  &\frac{|X^{d-1}_{i,+}|-N_{i,\ell/2}}{|X_{i,+}^{d-1}|}\cdot\Big(\big|\phi_\ell(T_{d,k};\sigma_0,\sigma_0)-\phi_\ell(T_{d,k};\sigma_0,\overline{\sigma_0})\big|+ (dk)^\ell\Big)\equiv \frac{C_{d,k,\ell}(|X^{d-1}_{i,+}|-N_{i,\ell/2})}{|X_{i,+}^{d-1}|}\,.
\end{aligned}
\end{equation}

Let $\varepsilon>0$, by \eqref{eq:weak_conv_proof_3}, Markov's inequality and the linearity of expectation
\[
\begin{aligned}
	&\BP\bigg(\bigg|\frac{\Phi_{i,\ell}}{|X_{i,+}^{d-1}|}-\big(\phi_\ell(T_{d,k};\sigma_0,\sigma_0)-\phi_\ell(T_{d,k};\sigma_0,\overline{\sigma_0})\big)\bigg|>\varepsilon\bigg)\\
	\leq &\BP\bigg(\frac{C_{d,k,\ell}(|X^{d-1}_{i,+}|-N_{i,\ell/2})}{|X_{i,+}^{d-1}|}>\varepsilon\bigg) \\
	\leq & \frac{C_{d,k,\ell}}{|X^{d-1}_{i,+}|\varepsilon}\cdot \BE[|X^{d-1}_{i,+}|-N_{i,\ell/2}]\\
	= & \frac{C_{d,k,\ell}}{|X^{d-1}_{i,+}|\varepsilon}\cdot \sum_{\sigma\in X^{d-1}_i}\BP\big(X_i(\sigma,\lceil \ell/2\rceil)\not\cong T_{d,k}(\lceil \ell/2\rceil)\big)\,.
\end{aligned}
\]
Using Theorem \ref{cor:local_convergence}, we conclude that 
\[
	\lim_{i\to\infty}\BP\bigg(\bigg|\frac{\Phi_{i,\ell}}{|X_{i,+}^{d-1}|}-\big(\phi_\ell(T_{d,k};\sigma_0,\sigma_0)-\phi_\ell(T_{d,k};\sigma_0,\overline{\sigma_0})\big)\bigg|>\varepsilon\bigg)=0\,,
\]
thus proving the convergence in probability. 

Next, we turn to deal with the convergence for the Laplacians. As for the adjacency matrices, it suffices by \eqref{eq:weak3} to prove that for every $\ell\geq 0$
\[	
	\binom{n_i}{d}^{-1}\mathrm{tr}(\Delta_{d-1}^+(X_i)^\ell)\underset{^{i\to\infty}}{\longrightarrow}\langle\Delta_{d-1}^+(T_{d,k})^\ell \one_{\sigma_0},\one_{\sigma_0}\rangle,\qquad\text{in probability}\,,
\]
or equivalently 
\[	
	\frac{1}{|X^{d-1}_+|}\sum_{\sigma \in X^{d-1}_+}\Big[\langle \Delta_{d-1}^+(X_i)^\ell \one_\sigma , \one_\sigma\rangle -\langle\Delta_{d-1}^+(T_{d,k})^\ell \one_{\sigma_0},\one_{\sigma_0}\rangle\Big]\underset{^{i\to\infty}}{\longrightarrow}0,\qquad\text{in probability}\,. 
\]
Note that $T_{d,k}$ is $k$-regular and hence that $\Delta_{d-1}^+(T_{d,k})=k\cdot \mathrm{Id}-A_{T_{d,k}}$. Furthermore, since $\Delta_{d-1}^+(X_i) = D_{X_i}-A_{X_i}$, where $D_{X_i}$ is the diagonal operator of the degrees, $D_{X_i}f(\sigma)=\deg(\sigma)f(\sigma)$, it follows that on the event $E_{i, \ell+1}(\sigma)=\{X_i(\sigma, \ell+1)\cong T_{d,k}( \ell+1)\}$
\[
	\Delta_{d-1}^+(X_i)=D_{X_i}-A_{X_i}=k\cdot \mathrm{Id}-A_{X_i}\,,
\]
and hence by the same argument used in the adjacency matrix case
\[
	\langle \Delta_{d-1}^+(X_i)^\ell \one_\sigma , \one_\sigma\rangle - \langle\Delta_{d-1}^+(T_{d,k})^\ell \one_{\sigma_0},\one_{\sigma_0}\rangle=0\,.
\]

In addition, since the degree of each of the $(d-1)$-faces in $X_i$ is bounded by $k$, it follows that (see \cite[Proposition 2.7(ii)]{PR12} for a similar argument)
\[
	|\langle \Delta_{d-1}^+(X_i)^\ell \one_\sigma , \one_\sigma\rangle - \langle\Delta_{d-1}^+(T_{d,k})^\ell \one_{\sigma_0},\one_{\sigma_0}\rangle|\leq 2((d+1)k)^\ell\,.
\]
Consequently 
\begin{equation}\label{eq:weak_conv_proof_3}
	\frac{1}{|X^{d-1}_+|}\sum_{\sigma \in X^{d-1}_+}\Big|\langle \Delta_{d-1}^+(X_i)^\ell \one_\sigma , \one_\sigma\rangle -\langle\Delta_{d-1}^+(T_{d,k})^\ell \one_{\sigma_0},\one_{\sigma_0}\rangle\Big|\leq \Big(1-\frac{N_{i,\ell+1}}{|X^{d-1}_{i,\pm}|}\Big)\cdot 2((d+1)k)^\ell\,,
\end{equation}
where $N_{i,\ell+1}$ is defined in \eqref{eq:The_definition_of_N}. The rest of the proof is similar to the proof for the adjacency operator. \hfill\qed


\section{The asymptotic number of simplicial spanning trees} \label{Number}	

\subsection{Proof of Theorem \ref{thm:main_Thm1}} \label{Weight-num}

	Let $X$ be a $d$-complex on $n$ vertices with a complete $(d-1)$-skeleton, and recall that the weighted number of $j$-dimensional SSTs is given by
\[ 
	\kappa_j(X)=\sum\limits_{T\in \mathcal{T}_j(X)} \big\vert \widetilde{H}_{j-1}(T;\mathbb{Z}) \big\vert^2\,.
\]  

We use the following version of the simplicial matrix tree theorem. 
\begin{Theorem}[\cite{DKM09}] \label{matrixtree-thm}
	Let $X$ be a $d$-dimensional simplicial complex. Denote by $\pi_d(x)$ the product of the non-trivial eigenvalues of the $(d-1)$-upper Laplacian. Then 
\[ 
	\pi_d(X)= \dfrac{\kappa_d(X) \cdot \kappa_{d-1}(X) }{\vert \tilde{H}_{d-2}(X;\mathbb{Z})\vert^2}\,. 
\]
\end{Theorem}

Furthermore, we recall Kalai's generalization of Cayley's formula, see \cite{Ka83}, which states
\[	
	\kappa_{d-1}(K_n^{(d-1)})=n^{ \binom{n-2}{d-1} }\,.
\]

Turning back to our setting, since $X$ has a complete $(d-1)$-skeleton, it follows that $|\tilde{H}_{d-2}(X;\mathbb{Z})|=1$. Furthermore we can apply Kalai's theorem to $X^{(d-1)}=K_n^{(d-1)}$, thus obtaining $\kappa_{d-1}(X)=n^{ \binom{n-2}{d-1} }$. Hence,
\[ 
	\pi_d(X) =\dfrac{\kappa_d(X) \cdot \kappa_{d-1}(X) }{\vert \tilde{H}_{d-2}(X;\mathbb{Z}) \vert ^2 }= n^{ \binom{n-2}{d-1} } \cdot \kappa_d(X)\,. 
\]
In particular, if $\pi_d(X)>0$, i.e., all zero-eigenvalues of $\Delta_{d-1}^+$ are trivial, by taking logarithm on both sides and dividing by $\binom{n}{d}$, the last equality can be rewritten as 
\begin{equation}\label{eq:num_of_SST_1}
\begin{aligned}
	\log \Big( \sqrt[\binom{n}{d}]{\kappa_d(X)} \Big) &= \binom{n}{d}^{-1}\sum\limits_{ \lambda \in \text{spec}(\Delta_{d-1}^+)\cap (0,\infty) } \log(\lambda) - \frac{d(n-d)\log(n)}{n(n-1)}\\
	&=\underset{(0,\infty)}{\int}  \log(t)d\mu_{\Delta_{d-1}^+(X)}(t)-\Theta\bigg(\frac{\log n}{n}\bigg)\,.
\end{aligned}
\end{equation}

Assume next that $(X_i)$ is a sequence of $d$-complexes with a complete $(d-1)$-skeleton such that $\pi_d(X_i)>0$ for every $i\geq 1$ and $|X_i^0|=n_i$ with $\lim_{i\to\infty}n_i=\infty$. If we were to know that
   \begin{equation} \label{integrals-conv}
   \underset{(0,\infty)}{\int}  \log(t)d\mu_{\Delta_{d-1}^+(X_i)}(t) \overset{i\rightarrow \infty}{\longrightarrow} \underset{(0,\infty)}{\int}  \log(t)d\nu_{d,k}(t)=:\overline{\xi}_{d,k},
   \end{equation}
in probability, then \eqref{eq:num_of_SST_1} would imply that 
\[
	\log \Big( \sqrt[\binom{n_i}{d}]{\kappa_d(X_i)} \Big) \underset{^{i\rightarrow \infty}}{\longrightarrow} \overline{\xi}_{d,k}\,,
\]
in probability and hence that 
\[
	\sqrt[\binom{n_i}{d}]{\kappa_d(X_i)} \underset{_{i\rightarrow \infty}}{\longrightarrow} e^{\overline{\xi}_{d,k}}\,
\]
in probability. Consequently, in order to complete the proof of Theorem \ref{thm:main_Thm1} is suffices to prove the following propositions.

\begin{proposition}\label{prop:num_of_SST_1}
   	Let $(X_i)_{i=1}^\infty$ be a sequence of $d$-dimensional $k$-regular uniform random Steiner complexes on $n_i$ vertices, with $(n_i)$ a sequence of $d$-admissible numbers satisfying $n_i\overset{i\rightarrow \infty}{\longrightarrow} \infty$. Then $\BP$-almost surely $\pi_d(X_i)>0$ for all sufficiently large $i$ and \eqref{integrals-conv} holds.
\end{proposition}

\begin{proposition}\label{prop:num_of_SST_2}
\[
	e^{\overline{\xi_{d,k}}}=\xi_{d,k}:=\dfrac{(k-1)^{k-1}}{\left( k-1-d \right)^{ \frac{k}{d+1}-1 } k^{\frac{d(k-1)-1}{d+1}} }\,.
\]
\end{proposition}
   

\subsection{Proof of Proposition \ref{prop:num_of_SST_1}}
	Recall that the spectrum of $\Delta_{d-1}^+(X)$ is contained in $[0,(d+1)k]$ for every $d$-complex whose degrees are uniformly bounded by $k$, c.f. \cite[Proposition 2.7(2)]{PR12}. Hence, for every $C\in (0,(d+1)k)$
\begin{equation}\label{eq:prop1}
	\int_{[C,\infty)}  \log(t)d\mu_{\Delta_{d-1}^+(X_i)}(t) = \int_{[C,(d+1)k]}  \log(t)d\mu_{\Delta_{d-1}^+(X_i)}(t) \,.
\end{equation}

If we assume in addition that $C\in (0,(\sqrt{k-1}-\sqrt{d})^2)$, then the function $\log t$ is continuous and bounded in $[C,(d+1)k]$ and in addition, by Theorem \ref{Rosen},
\[
	\mathrm{supp}(\mu_{T_{d,k}}) \subset [(\sqrt{k-1}-\sqrt{d})^2,(\sqrt{k-1}+\sqrt{d})^2]\subset [C,k(d+1)]\,.
\] 
Hence by Theorem \ref{thm:main_Thm2}
\begin{equation}\label{eq:prop2}
	\int_{[C,(d+1)k]}  \log(t)d\mu_{\Delta_{d-1}^+(X_i)}(t) \overset{i\rightarrow \infty}{\longrightarrow} \int_{[C,(d+1)k]}  \log(t)d\nu_{d,k}(t)=\int_{(0,\infty)}  \log(t)d\nu_{d,k}(t) = \overline{\xi}_{d,k}\,,
\end{equation}
in probability. 
   
Consequently, if we can find $C\in (0,(\sqrt{k-1}-\sqrt{d})^2)$ such that all non-trivial eigenvalues of $\Delta_{d-1}^+(X_i)$ are within $[C,\infty)$ for all sufficiently large $i$ $\BP$-almost surely, then  \eqref{eq:prop1} and \eqref{eq:prop2} would give
\[
	\int_{(0,\infty)}  \log(t)d\nu_{X_i}(t) \overset{i\rightarrow \infty}{\longrightarrow} \int_{(0,\infty)}  \log(t)d\nu_{T_{d,k}}(t) = \overline{\xi}_{d,k}\,,	
\]
and in addition we would get $\pi_d(X_i)>0$ for all sufficiently large $i$, thus completing the proof of Proposition \ref{prop:num_of_SST_1}. 

In order to prove the above, we first recall the following result by Abu-Fraiha and Meshulam.
\begin{Theorem}[\cite{abu2017homol}]\label{spec-gap-prob}
   	Let $(X_i)_{i=1}^\infty$ be a sequence of $d$-dimensional, $k$-regular uniform random Steiner complexes on $n_i$ vertices, with $(n_i)$ a sequence of $d$-admissible numbers satisfying $n_i\overset{i\rightarrow \infty}{\longrightarrow} \infty$. For $\varepsilon> 0$, define the event 
\[ 
   	E_i^{\varepsilon}:= \big\{ \text{all non-trivial eigenvalues of } A_{X_i}\; \text{are contained in } (-\infty,2d\sqrt{k-1}+\varepsilon]\big \}^c. 
\]
   	If $k>(d+1)^2+1$, then $\mathbb{P}\big(  E_i^{\varepsilon}  \; \text{eventually} \big)=1.$ 
   \end{Theorem}

Since the statement provided above for Abu-Fraiha's result is slightly different than the one stated in \cite{abu2017homol}, we provide a revised proof in the Appendix.

Theorem \ref{spec-gap-prob} provides us with an upper bound on the eigenvalues of the adjacency matrix of $X_i$ for all sufficiently large $i$, while we are interested in a lower bound on the eigenvalues of $\Delta_{d-1}^+(X_i)$. These operators are related to one another via the relation $\Delta_{d-1}^+(X_i)=D_{X_i}-A_{X_i}$, where $D_{X_i}$ is the degree operator defined by $D_{X_i}f(\sigma)=\deg_{X_i}(\sigma)f(\sigma)$ for all $f\in \Omega^{d-1}(X_i)$. If we were to know that $X_i$ is $k$-regular, then $D_{X_i}=k\mathrm{Id}$ and the relation between the eigenvalues would have been trivial. However, in general it is not true that the resulting complex $X_i$ is $k$-regular. That being said, the following claim shows that with high probability the degrees are between $k-d-1$ and $k$. 

\begin{Claim}
	Let $X$ be a $(d,k,n)$-uniform random Steiner complex composed of the independent random $(n,d)$-Steiner systems $S_1,\ldots,S_k$ and assume that $k\geq d+1$. Then there exists $C_{d,k}\in (0,\infty)$ such that for every $1\leq j\leq k-1$
	\[
		\BP\big(\exists \sigma\in X^{d-1} \text{ such that }\deg(\sigma)\leq k-j\big) \leq \frac{C_{d,k}}{n^{j-d}}\,.
	\]	
\end{Claim}

Note that the inequality yields a trivial bound whenever $j\leq d$.

\begin{proof}
	By a union bound, it suffices to prove that $\BP(\deg(\sigma)\leq k-j)\leq C_{d,k}n^{-j}$ for any $(d-1)$-face $\sigma$. To this end fix $\sigma\in X^{d-1}$ and for $1\leq i\leq k$ denote by $\tau^\sigma_i$ the unique $d$-face in $S_i$ containing $\sigma$. Then 
\[
	\BP(\deg(\sigma)\leq k-j) = \BP(|\{\tau^\sigma_{1},\ldots,\tau^\sigma_{k}\}|\leq  k-j) = \sum_{i=1}^{k-j}\BP(|\{\tau^\sigma_1,\ldots,\tau^\sigma_k\}|=i)\,.
\]
For each $1\leq i\leq k-j$, the probability $\BP(|\{\tau^\sigma_1,\ldots,\tau^\sigma_k\}|=i)$, can be written more explicitly as 
\[
	\BP(|\{\tau^\sigma_1,\ldots,\tau^\sigma_k\}|=i) = \sum_{\substack{B_1,\ldots,B_i\subset [k] \\ \uplus_{m=1}^i B_m =[k]\\ B_m\neq\emptyset ~\forall 1\leq m\leq i}}\sum_{\substack{v_1,\ldots,v_i\in [n]\setminus \sigma \\ \text{distinct}}}\BP(\tau_j^\sigma=v_m\sigma ~\forall 1\leq m\leq i~ \forall j\in B_m)\,.
\]
Noting that from the independence of the Steiner systems and Proposition \ref{Ten-Prop1}, for every partition $(B_1,\ldots,B_i)$ of $[k]$ into non-empty sets and every choice of $v_1,v_2,\ldots,v_i\in [n]\setminus \sigma$
\[
	\BP(\tau_j^\sigma=v_m\sigma ~\forall 1\leq m\leq i~ \forall j\in B_m) = \prod_{m=1}^i\prod_{j\in B_m}\BP(\tau_j^\sigma=v_m\sigma) \leq \prod_{m=1}^i\prod_{j\in B_m}\Big(\frac{C_d}{n}\Big) = \Big(\frac{C_d}{n}\Big)^{k}\,,
\]
it follows that 
\[
\begin{aligned}
	& \BP(|\{\tau_1^\sigma,\ldots,\tau_k^\sigma\}|=i)\\
	=&\sum_{\substack{B_1,\ldots,B_i\subset [k] \\ \uplus_{m=1}^i B_m =[k]\\ B_m\neq\emptyset ~\forall 1\leq m\leq i}}\sum_{\substack{v_1,\ldots,v_i\in [n]\setminus \sigma \\ \text{ distinct}}}\BP(\tau_j=v_m\sigma~\forall 1\leq m\leq i~\forall j\in B_m) \leq i^k(n-d)^i\Big(\frac{C_d}{n}\Big)^{k}\,.
\end{aligned}
\]
Summing over $i$ from $1$ to $k-j$ gives
\[
	\BP(|\{\tau_1^\sigma,\ldots,\tau_k^\sigma\}|\leq k-j) \leq \sum_{i=1}^{k-j} i^k(n-d)^i\Big(\frac{C_d}{n}\Big)^{k}\leq k^{k+1}n^{k-j}\Big(\frac{C_d}{n}\Big)^{k}= \frac{C_{d,k}}{n^{j}}\,,
\] 
as required. 
\end{proof}

The last claim together with Borel-Cantelli, immediately gives
\begin{Cor}\label{cor:Steiner_deg}
	Let $(X_i)_{i\geq 1}$ be a sequence of $(d,k,n_i)$-uniform random Steiner complexes with $(n_i)$ a sequence of $d$-admissible numbers such that $\lim_{i\to\infty}n_i=\infty$. Then 
\[
	\BP(\exists \sigma\in X_i^{d-1} ~~\deg(\sigma)< k-d-1  \text{ for infinitely many }i)=0\,.
\]
\end{Cor}

Combining Theorem \ref{spec-gap-prob} and Corollary \ref{cor:Steiner_deg} we conclude that for every $\varepsilon>0$, the event 
\[
	B_\varepsilon=\Bigg\{\begin{array}{c}
		\text{For sufficiently large }i \text{ all non-trivial eigenvalues of the adjacency operator} A_{X_i} 
		\text{ are}\vspace{-0.25cm}\\\text{within }(-\infty,2d\sqrt{k-1}+\varepsilon] \text{ and }\deg(\sigma) \text{ is between } k-d-1 \text{ and } k \text{ for all } \sigma\in X_i^{d-1}
	\end{array}\Bigg\}\,,
\]
has probability $1$ provided $k> (d+1)^2+1$. 

Finally, note that on the event $B_\varepsilon$, all the eigenvalues of the operator $k\mathrm{Id}-A_{X_i}$ are within $(k-2d\sqrt{k-1}-\varepsilon,\infty)$. Furthermore, the norm of the difference between the operators $k\mathrm{Id}-A_{X_i}$ and $\Delta_{d-1}^+(X_i)$, i.e. the norm of the operator $k\mathrm{Id}-D_{X_i}$ is bounded by $d+1$, since $k\mathrm{Id}-D_{X_i}$ is a diagonal operator with entries in $\{0,1,2,\ldots,d+1\}$. Hence, by Weyl's inequalities (c.f. \cite{Tao10}), we conclude that for all large enough $i$ all non-trivial eigenvalues of $\Delta_{d-1}^+(X_i)$ are within $[k-2d\sqrt{k-1}-d-1-\varepsilon,\infty)$. Finally, since $k-2d\sqrt{k-1}-d-1>0$, whenever 
\[
	k>2d^2\Big(1+\sqrt{1+\frac{1}{d}}\Big)+d+1\geq 4d^2+d+2\,,
\]
the result follows by taking $C\in (0,k-2d\sqrt{k-1}-d-1)$. \hfill\qed


\subsection{Chebyshev polynomials} \label{Cheby}
	The proof of Proposition \ref{prop:num_of_SST_2} is based on Chebyshev's polynomials, whose definition and helpful properties are summarized in this subsection. 
	
\begin{defn} Chebyshev polynomials of the first kind are defined as the unique sequence of polynomials $(T_n)_{n=0}^\infty$ satisfying $\deg(T_n)=n$ for all $n\in \mathbb{N}_0$ and $T_n\circ \cos(x)=\cos(nx)$ for all $n\in \mathbb{N}_0$.
\end{defn} 

Chebyshev's polynomials are classical and well-studied, c.f. \cite{MH03}. Below we collect several useful properties they possess. 

{\bf Orthogonality} 
   \begin{equation}\label{Cheb-property1}
  	 \int_{-1}^{1} \frac{T_n(x)T_m(x)}{\sqrt{1-x^2}}dx= \begin{cases}
		   0 & \text{if } n\neq m\\
		   \pi & \text{if }n,m=0 \\
		   \frac{\pi}{2} & \text{if }n=m\neq 0
	   \end{cases}\,.
   \end{equation}

{\bf Logarithmic generating function} For all $|t|<1$ and $x\in\BR$ 
   \begin{equation}\label{Cheb-property2}
	   \log(1-2xt+t^2)=-2\sum\limits_{n=1}^\infty T_n(x)\cdot \frac{t^n}{n}\,.
   \end{equation}

{\bf Expansion of powers via Chebyshev's polynomials} For every $n\geq 0$
\begin{equation}\label{Cheb-property3}
	x^{2n+1}= 2^{-2n}\sum\limits_{m=0}^n \binom{2n+1}{n-m} T_{2m+1}(x)
\end{equation}
and 
\begin{equation}\label{Cheb-property4}
	x^{2n}=2^{1-2n} \sum\limits_{m=1}^n \binom{2n}{n-m} T_{2m}(x)+ 2^{-2n} \binom{2n}{n}\,. 
\end{equation}
Hence for every converging power series 
\begin{equation}\label{Cheby3}  
\begin{aligned} 
   	\sum_{n=0}^\infty c_n x^n & = \Bigg( \sum\limits_{n=0}^\infty \frac{c_{2n}}{2^{2n}} \binom{2n}{n} \Bigg) T_0(x)\\ &+ \sum\limits_{m=0}^\infty \Bigg( \sum\limits_{n=m}^\infty \frac{c_{2n+1}}{2^{2n}} \binom{2n+1}{n-m}  \Bigg) T_{2m+1}(x)+ \sum\limits_{m=1}^\infty 2\Bigg( \sum\limits_{n=m}^\infty \frac{c_{2n}}{2^{2n}} \binom{2n}{n-m}  \Bigg) T_{2m}(x)\,. 
\end{aligned}
\end{equation}

	Let $g:[-1,1]\to \BR$ be a continuous function. Denoting $h(x)=g(x)\sqrt{1-x^2}$, the orthogonality property \eqref{Cheb-property1} with respect to the function $(1-x^2)^{-1/2}$, enables us to develop $h$ as a power series in Chebyshev polynomials 
   \[ 
   		h(x)=\sum\limits_{n=0}^\infty \alpha_n T_n(x)\,, 
   	\]
where $\alpha_n$ is given by
   \[ 
   \alpha_n= \begin{cases}
	   \frac{1}{\pi} \int_{-1}^1 T_0(x) g(x) & , \text{if} \; n=0\\
	   &\\
	   \frac{2}{\pi} \int_{-1}^{1}T_n(x)g(x)dx & , \text{if } n\geq 1
	   \end{cases}\,.
	 \]

   In particular, if the Chebyshev power series of $h$ converges uniformly on $(-1,1)$ we get from  \eqref{Cheb-property2} and integration term by term that
   \begin{equation}\label{eq:Cheb_4}
    \begin{aligned}
    \int_{-1}^1 \log(1-2xt+t^2) \cdot g(x) dx&=\int_{-1}^1 \left(-2 \sum\limits_{n=1}^\infty T_n(x)\frac{t^n}{n}  \right)\cdot \frac{h(x)}{\sqrt{1-x^2}}dx \\
     &=-2 \sum\limits_{n=1}^\infty \frac{t^n}{n} \int_{-1}^{1} \frac{T_n(x)h(x)}{\sqrt{1-x^2}}dx
     =- \pi \sum\limits_{n=1}^\infty \frac{\alpha_n}{n}t^n.
    \end{aligned}
   \end{equation}


\subsection{Proof of Proposition \ref{prop:num_of_SST_2}}

Recall that 
\[
	\overline{\xi}_{d,k}=\int_{(0,\infty)}\log(t)d\nu_{d,k}(t)\,.
\]
By Theorem \ref{Rosen}, for $k\geq d+1$, we can rewrite the last expression as 
\[
	\overline{\xi}_{d,k} =\int_{I_{d,k}}  \log(t)\frac{k\sqrt{4(k-1)d-(k-1+d-t)^2}}{2\pi t((d+1)k-t)}dt\,,
\]
where
\[ 
	I_{d,k}=\big[(\sqrt{k-1}-\sqrt{d})^2,(\sqrt{k-1}+\sqrt{d})^2]\,.
\]

Denoting $\omega:=2\sqrt{d(k-1)}$ and using the change of variables $x=(k-1+d-t)/\omega$ gives 
\[
	\overline{\xi}_{d,k} = \int_{-1}^1 \log(k-1+d-\omega x)\frac{k\omega^2\sqrt{1-x^2}}{2\pi (k-1+d-\omega x)(d(k-1)+1+\omega x)}dx\,.
\]

Defining $g_{d,k}:[-1,1]\to \BR$ by 
\[ 
	g_{d,k}(x)=  \frac{k\omega^2\sqrt{1-x^2} }{2\pi(k-1+d-\omega x)(d(k-1)+1+\omega x)}\,,
\]  
and noting that $\int_{-1}^1 g_{d,k}(x)dx=1$, for $k\geq d+1$, we can write 
\begin{equation}\label{eq:computation1}
\begin{aligned}
	\overline{\xi}_{d,k} & = \int_{-1}^1 \log(k-1+d-\omega x)g_{d,k}(x)dx \\
	& = \int_{-1}^1 \Big[\log(k-1+d) +\log\Big(1-\frac{\omega x}{k-1+d}\Big)\Big]g_{d,k}(x)dx\\
	& = \log(k-1+d) +\int_{-1}^1 \log\Big(1-\frac{\omega x}{k-1+d}\Big) g_{d,k}(x)dx\,.
\end{aligned}	
\end{equation}

	The remaining integral is computed using Chebyshev polynomials. Denoting $h_{d,k}(x) = g_{d,k}(x)\cdot \sqrt{1-x^2}$ and assuming that $h_{d,k}(x)=\sum_{n=0}^\infty \alpha_nT_n(x)$, we conclude from \eqref{eq:Cheb_4} that for all $|t|<1$
\begin{equation} \label{eq:computation2}
	\int_{-1}^1 \log(1-2xt+t^2) \cdot g_{d,k}(x)dx =- \pi \sum\limits_{n=1}^\infty \frac{\alpha_n}{n}t^n\,.
\end{equation}

Furthermore, since 
\[
	\log(1-2xt+t^2)=\log(1+t^2) + \log\Big(1-\frac{2t}{1+t^2}x\Big)\,,
\]
and recalling that in our case $\int_{-1}^1g_{d,k}(x)dx=1$, we conclude that 
\begin{equation}\label{eq:computation3}
	\int_{-1}^1 \log(1-2xt+t^2) \cdot g_{d,k}(x)dx = \log(1+t^2) + \int_{-1}^1 \log\Big(1-\frac{2t}{1+t^2}x\Big) \cdot g_{d,k}(x)dx\,.
\end{equation}

Combining \eqref{eq:computation2} and \eqref{eq:computation3} and taking $|t|<1$ such that $\frac{\omega}{k-1+d}= \frac{2t}{1+t^2}$, namely 
\begin{equation} \label{t-def}
	t= \frac{ 1-\sqrt{1-(\frac{\omega}{k-1+d})^2} }{\frac{\omega}{k-1+d}}\,,
\end{equation}
we conclude that 
\begin{equation}\label{eq:computation4}
	\overline{\xi}_{d,k} = \log(k-1+d)-\log(1+t^2) -\pi \sum_{n=1}^\infty \frac{\alpha_n}{n}t^n\,.
\end{equation}
Hence, all that remains is to find the coefficients $\alpha_n$ in the Chebyshev expansion of $h_{d,k}$. To this end, note that partial fractions and Taylor expansion for the function $\frac{1}{1+ \alpha x}$ for $\alpha\in\BR$, gives for all $|x|<\max\{\frac{k-1+d}{\omega},\frac{d(k-1)+1}{\omega}\}$
\begin{equation*}
\begin{aligned}
	h_{d,k}(x)& =\frac{k\omega^2(1-x^2)}{2\pi(k-1+d-\omega x)(d(k-1)+1+\omega x)} \\
	&= \frac{1}{2\pi(d+1)}\cdot \frac{1}{\Fe_1\Fe_2} + \frac{1}{2\pi(d+1)}\cdot \frac{\omega(\Fe_2-\Fe_1)}{\Fe_1^2\Fe_2^2}x\\
	&+\frac{1}{2\pi(d+1)} \sum\limits_{n=2}^\infty \Bigg(\frac{\omega^2-\Fe_1^2}{\Fe_1^2}\bigg(\frac{\omega}{\Fe_1}\bigg)^n+(-1)^n\frac{\omega^2-\Fe_2^2}{\Fe_2}\bigg(\frac{\omega}{\Fe_2}\bigg)^n\Bigg)x^n\,,
\end{aligned}
\end{equation*}
where we introduced the notation 
\begin{equation}
	\mathfrak{e}_1:=k-1+d \qquad \text{and} \qquad \mathfrak{e}_2:=d(k-1)+1 \,.
\end{equation}
Also, note that for $k,d\in\BN$ such that $k\geq d+1$, we have $\max\{\frac{k-1+d}{\omega},\frac{d(k-1)+1}{\omega}\}>1$ and thus the expansion is valid for all $x\in [-1,1]$. 
  
Using \eqref{Cheby3} to transform the power series into a Chebyshev series, we conclude that for $m\geq 0$
\begin{equation*}
\begin{aligned}
	\alpha_{2m+1}= \frac{1}{2\pi(d+1)}\sum\limits_{n=m}^\infty \Bigg( \frac{\omega^2-\Fe_1^2}{\Fe_1} \cdot \Big( \frac{\omega}{2\Fe_1} \Big)^{2n+1}- \frac{\omega^2-\Fe_2^2}{\Fe_2} \cdot \Big( \frac{\omega}{2\Fe_2} \Big)^{2n+1} \Bigg) \binom{2n+1}{n-m}\,,
   \end{aligned}
   \end{equation*}
and for $m\geq 1$
\begin{equation*}
\begin{aligned}
	\alpha_{2m}	= \frac{1}{2\pi(d+1)}\sum\limits_{n=m}^\infty \Bigg( \frac{\omega^2-\Fe_1^2}{\Fe_1} \cdot \Big( \frac{\omega}{2\Fe_1} \Big)^{2n}+ \frac{\omega^2-\Fe_2^2}{\Fe_2} \cdot \Big( \frac{\omega}{2\Fe_2} \Big)^{2n} \Bigg)\binom{2n}{n-m}\,.
\end{aligned}
\end{equation*}
  
Next, using the identity, c.f. \cite[Chapter 2.5]{Wi06},
\[ 
	\frac{1}{ \sqrt{1-4z} } \Big( \frac{1-\sqrt{1-4z}}{2z}  \Big)^k= \sum\limits_{n=0}^\infty \binom{2n+k}{n} z^n\,,
\]
and the abbreviation 
\[
	r_1 = \frac{\Fe_1-\sqrt{\Fe_1^2-\omega^2}}{\omega}=\frac{2d}{\omega} \qquad\text{and}\qquad r_2 = \frac{\Fe_2-\sqrt{\Fe_2^2-\omega^2}}{\omega}=\frac{2}{\omega}\,,
\]
we conclude that for $m\geq 0$, 
\begin{equation}\label{eq:computation5}
	\alpha_{2m+1}= -\frac{\sqrt{\Fe_1^2-\omega^2}}{\pi(d+1)}r_1^{2m+1}+ \frac{\sqrt{\Fe_2^2-\omega^2}}{\pi(d+1)}r_2^{2m+1} \,,
\end{equation}
and for $m\geq 1$, 
\begin{equation}\label{eq:computation6}
	\alpha_{2m}= -\frac{\sqrt{\Fe_1^2-\omega^2}}{\pi(d+1)}r_1^{2m}- \frac{\sqrt{\Fe_2^2-\omega^2}}{\pi (d+1)}r_2^{2m}\,.
\end{equation}
 
Combining \eqref{eq:computation4}, \eqref{eq:computation5} and \eqref{eq:computation6} together with the fact that $\sum_{n=1}^\infty \frac{x^{2n}}{2n}= -\frac{1}{2}\log(1-x^2)$ and $\sum_{n=0}^\infty \frac{x^{2n+1}}{2n+1}=\frac{1}{2}\log\big(\frac{1+x}{1-x}\big)$ for $|x|<1$, we conclude that as long as $|r_1|,|r_2|<1$ (which is the case whenever $k\geq d+1$)
\[
\begin{aligned}
	\overline{\xi}_{d,k} &= \log(k-1+d)-\log(1+t^2) -\pi \sum_{n=1}^\infty \frac{\alpha_n}{n}t^n\\
	& = \log(k-1+d)-\log(1+t^2)-\frac{\sqrt{\Fe_1^2-\om^2}}{d+1}\log(1-r_1t) - \frac{\sqrt{\Fe_2^2-\om^2}}{d+1}\log(1+r_2t) \\
\end{aligned}
\] 
 
Plugging in the values  
\[
	\Fe_1=k-1+d,\quad \Fe_2=d(k-1)+1,\quad t=r_1=\frac{2d}{\omega},\quad r_2=\frac{2}{\omega},\quad \omega = 2\sqrt{(k-1)d}\,,
\] 
gives 
\[
	\overline{\xi}_{d,k} = \log\Bigg(\frac{(k-1)^{k-1}}{(k-1-d)^{\frac{k-1-d}{d+1}}k^{\frac{d(k-1)-1}{d+1}}}\Bigg)\,,
\]
as required. \hfill\qed


\section{Open problems and conjectures} \label{future}

\subsection{Sampling of $(n,d)$-Steiner complexes}
Following the discussion in Section \ref{sec:Results} we suggest the following:  
\begin{problem}
Find an (efficient) algorithm for sampling $(n,d)$-Steiner systems uniformly at random. 
\end{problem}

\subsection{Uniform random Steiner complexes and the matching model}
When sampling $k$ independent random matchings uniformly at random on $n$ vertices, it is known (see \cite{BC78}) that the probability for obtaining a simple graph converges to $something$ as $n$ tends to infinity, and that conditioned on obtaining a simple graph, the resulting distribution is uniform over all such graphs. 
\begin{question}
	Is there an analogue of the above result in higher-dimensions?
\end{question}

\subsection{Improving the regularity threshold}
	The regularity condition in Theorem \ref{thm:main_Thm1} requires that $k>k(d)\equiv 4d^2+d+2$, however this condition only arises from Theorem \ref{spec-gap-prob} which in turn follows from applying Garland's method. Except for the restriction arising from Theorem \ref{spec-gap-prob}, the only requirement is  that $k> d+1$. A natural question arises as to whether the threshold on $k$ is indeed strict, or whether a finer analysis would yield a better threshold.
	
	Since the number of $d$-faces in a $k$-regular $d$-complex $X$ on $n$ vertices is $\binom{k}{d+1}\binom{n}{d}$, it follows that for $k<d+1$, the number of $d$-faces is strictly smaller than the number of $(d-1)$-faces $\binom{n}{d}$. Hence, $\Delta_{d+1}^+(X)$ always have a non-trivial $0$ eigenvalue, and thus by the simplicial matrix tree theorem $\kappa_d(X)=0$. Consequently, the least lower bound on $k$ for Theorem \ref{thm:main_Thm1} is $d$. A similar, yet slightly more evolved argument would show  that in fact the least lower bound is $k=d+1$. This leads us to the following conjecture. 
     
\begin{conjecture}
	The condition on $k$ in Theorem \ref{thm:main_Thm1} can be improved to $k>d+1$. 
\end{conjecture}

\subsection{Unweighted asymptotic number of simplicial spanning trees}
     Our result deals with a weighted number of simplicial spanning trees. One would hope to generalize Mckay's results for the number of spanning trees in a graph by proving an unweighted version of Theorem \ref{thm:main_Thm1}. See \cite{LP19} for partial results in this direction in the case of the complete $d$-complex. 
     
\subsection{Finer asymptotic for the number of SSTs}
	In Theorem \ref{thm:main_Thm1}, it is shown that $\sqrt[ \binom{n_i}{d} ]{ \kappa_d(X_i) }  \overset{i\rightarrow \infty}{\longrightarrow} \xi_{d,k}$, namely, $ \kappa_d(X_i) =\xi_{d,k}^{\binom{n_i}{d}(1+o(1))}$ as $i\to\infty$. It would be interesting to obtain better bounds on $\kappa_d(X_i)$. For example can one say something about the next order of $\kappa_d(X_i)$ by studying the sequence $\kappa_d(X_i)\xi_{d,k}^{-\binom{n_i}{d}}$?
	
\subsection{Asymptotic number of SSTs in other random sampling models}
	In this work we studied the asymptotic weighted number of SST's in simplicial complexes sampled from uniform random Steiner complexes. One can hope that similar methods can be used to study other models.

\begin{appendices}

\section{Proof of Theorem \ref{spec-gap-prob}}\label{APP1}

Let us start by stating a result of Friedmann regarding the spectral gap in the matching model. 
\begin{Theorem}[\cite{Fri08}] \label{Friedmann1}
	Fix $k\in \mathbb{N}$ and $\varepsilon>0$. Then there exists a constant $C_{k,\varepsilon}\in (0,\infty)$ such that a random graph $G$ on $n$ vertices sampled according to the matching model satisfies
\begin{equation}\label{eq:Friedmann}
	\mathbb{P}\Big( \vert \lambda_i(G)\vert \leq 2\sqrt{k-1} +\varepsilon \Big)\geq 1- \frac{C_{k,\varepsilon}}{n^{\tau(k)}},\qquad \forall 2\leq i\leq n\,,
\end{equation}
where $\lambda_1(G)\geq ...\geq \lambda_n(G)$ are the eigenvalues of $A(G)$, the adjacency matrix of $G$, and $ \tau(k)= \lceil  \sqrt{k-1} \rceil -1 $. Furthermore, there exists a constant $C_k>0$, such that 
\[ 
	\mathbb{P}\Big(  \lambda_2(G) > 2\sqrt{k-1} \Big)\geq \frac{C_k}{n^{s(k)}}\,, 
\]
   	where $s(k)= \lfloor \sqrt{k-1} \rfloor$.
   \end{Theorem}

	Let $\epsilon>0$. As stated before in Theorem  \ref{Friedmann1}, a random graph $G$ on $n$ vertices sampled according to the matching model satisfies \eqref{eq:Friedmann}.	Recall that for every $\sigma \in X_i^{d-2}$, the link of $\sigma$, denoted $\mathrm{lk}(X_i,\sigma)$ is a random graph on $n_i-d+1$ vertices, distributed according to the matching model with parameter $k$ and therefore, the event 
\[
	E_{i,\sigma,\varepsilon} = \{\lambda_2(\mathrm{lk}(X_i,\sigma))>2\sqrt{k-1}+\varepsilon\}
\]
satisfies 
\[
	\BP(E_{i,\sigma,\varepsilon})\leq C_{k,\varepsilon}(n_i-d+1)^{-(\lceil\sqrt{k-1}\rceil-1)}\,.
\]	

A union bound, thus gives
\[
	\BP\bigg(\bigcup_{\sigma\in X^{d-2}_i} E_{i,\sigma,\varepsilon}\bigg) \leq |X^{d-2}_i|\cdot C_{k,\varepsilon}(n_i-d+1)^{-(\lceil\sqrt{k-1}\rceil-1)}\leq C_{k,\varepsilon}n_i^{d-\lceil\sqrt{k-1}\rceil}\,,
\]
where in the last bound we used the fact that $|X^{d-2}_i|=\binom{n_i}{d-1}$.

By Garland's method (c.f. \cite{GW14}), on the event $\bigcup_{\sigma\in X^{d-2}_i} E_{i,\sigma,\varepsilon}$, all non-trivial eigenvalues of the adjacency matrix of $X_i$ are within $(-\infty,2d\sqrt{k-1}+\varepsilon)$.

Hence, whenever $\lceil \sqrt{k-1} \rceil >d+1$, by the Borel-Cantelli lemma, only finitely many of the random complexes $X_i$ do not satisfy 
 	\[ 
 		\text{spec}(A_{X_i}) \subset (-\infty,2d\sqrt{k-1}+\varepsilon)\,.
 	\]
Since $\lceil\sqrt{k-1}\rceil >d+1$ whenever $k>(d+1)^2+1$, the result follows.

\end{appendices}

\bibliography{Biblio}
\bibliographystyle{alpha}

$~$\\
Department of mathematics,\\
Technion - Israel Institute of Technology\\
Haifa, 3200003, Israel.\\
E-mail: ron.ro@technion.ac.il	\\
E-mail: tenen25@campus.technion.ac.il	
\end{document}